\documentclass[sn-mathphys-num]{sn-jnl}

\usepackage{graphicx}%
\usepackage{multirow}%
\usepackage{amsmath,amssymb,amsfonts}%
\usepackage{amsthm}%
\usepackage{mathrsfs}%
\usepackage[title]{appendix}%
\usepackage{xcolor}%
\usepackage{textcomp}%
\usepackage{manyfoot}%
\usepackage{booktabs}%
\usepackage{algorithm}%
\usepackage{algorithmicx}%
\usepackage{algpseudocode}%
\usepackage{listings}%
\usepackage{mathtools}
\usepackage{tikz}
\usepackage{amscd} 
\usetikzlibrary{cd}
\mathtoolsset{showonlyrefs}

\numberwithin{equation}{section}

\theoremstyle{thmstyleone}%
\newtheorem{theorem}{Theorem}[section]
\newtheorem{proposition}[theorem]{Proposition}%
\newtheorem{lemma}[theorem]{Lemma}%
\newtheorem{cor}[theorem]{Corollary}

\theoremstyle{definition}%
\newtheorem{definition}[theorem]{Definition}%
\newtheorem{example}[theorem]{Example}%
\newtheorem{remark}[theorem]{Remark}%

\begin{document}

\title[Doubly Totally-Umbilical Statistical Submanifolds in the Probability Simplex]{Doubly Totally-Umbilical Statistical Submanifolds in the Probability Simplex}


\author{\fnm{Ryu} \sur{Ueno}}



\abstract{We give a complete classification of doubly totally-umbilical submanifolds in the probability simplex.
The probability simplex is one of the most standard statistical manifolds, and the classical information geometry initiated by S. Amari and H. Nagaoka is the statistical submanifold theory of the probability simplex.
Autoparallel submanifolds in the probability simplex with respect to the e-connection or the m-connection are the main focus in Information geometry, since they are characterized as important probability distribution families.
On the other hand, H. Furuhata defined doubly totally-umbilical submanifolds in statistical manifolds inspired by the surface theory of the Euclidean space.
In this study, we were able to give complete classification of the doubly totally-umbilical submanifolds in the probability simplex.
We first introduce new results about basic aspects on doubly totally-umbilical submanifolds, and then give the complete classification in the probability simplex.}




\maketitle

\section{Introduction}\label{sec1}
The purpose of this paper is to initiate the submanifold theory of the probability simplex, and we start by classifying doubly totally-umbilical submanifolds.
Classical information geometry developed by S. Amari and H. Nagaoka studies the differential geometry of statistical models as submanifolds in the probability simplex \cite{AmariNagaoka2000}.
The probability simplex $(\Delta^n,g^F,\nabla^{(\mathrm{e})})$ consists of the Fisher metric $g^F$ and the exponential connection $\nabla^{(\mathrm{e})}$ on the set of positive probability distributions $\Delta^n=\{p:\{1,\ldots,n+1\}\to(0,1)\mid \sum_{\omega=1}^{n+1}p(\omega)=1\}$ which is a manifold.
The tools $g^F,\nabla^{(\mathrm{e})}$ are induced by Chentsov’s theorem under the invariance of Markov embeddings.
Moreover, the Riemannian manifold $(\Delta^n,g^F)$ is isometric to an open set of the Euclidean sphere of radius $2$, thus its sectional curvature is constant equal to $\frac{1}{4}$.
On $\Delta^n$, another affine connection called the mixture connection $\nabla^{(\mathrm{m})}$ is defined by the following equation$:$
\begin{align}
    \label{conjeq}
    Xg^F(Y,Z) = g^F\left(\nabla^{(\mathrm{e})}_X Y,Z\right) + g^F\left(Y,\nabla^{(\mathrm{m})}_X Z\right),\quad X,Y,Z\in\Gamma(TM). 
\end{align}
A submanifold $M\subset\Delta^n$ is called a \textit{statistical model}, or a \textit{family of probability distributions}.
In particular, it is known that $M$ is an \textit{exponential family} if and only if it is $\nabla^{(\mathrm{e})}$-autoparallel, and $M$ is a \textit{mixture family} if and only if it is $\nabla^{(\mathrm{m})}$-autoparallel.
While these two classes have been extensively studied, much less is known about more general submanifolds in the probability simplex \cite{AmariNagaoka2000}.\\

On the other hand, the geometry of statistical manifolds studies the differential-geometric structure arising in information geometry.
A pair $(g,\nabla)$ of a Riemannian metric $g$ and a torsion-free affine connection $\nabla$ on a manifold $M$ is called a \textit{statistical structure} if $\nabla g$ is a totally symmetric tensor field, and here the triplet $(M,g,\nabla)$ is called a \textit{statistical manifold}.
Another affine connection called the \textit{conjugate connection} $\nabla^*$ is defined by the equation analogous to \eqref{conjeq}.
H. Furuhata introduced doubly totally-umbilical submanifolds as the natural analogues of totally-umbilical submanifolds in Riemannian geometry \cite{Furuhata2024Toward}.
Totally-umbilical submanifolds have been extensively studied and continue to attract considerable attention in Riemannian geometry \cite{Usiraj,Chen_1980,JIMENEZ2022101862,Siddesha2019,Chen2019}, and in pseudo-Riemannian geometry \cite{b47c6540-6bcd-33b7-ac07-4e5c2b947949,10.21099}.\\

The definition of doubly totally-umbilical submanifolds is given in Defition \ref{defum}.
We show that the doubly totally-umbilical submanifolds in the probability simplex are described as follows:\\

\noindent\textbf{Theorem \ref{maint}.} A submanifold $M^m$ of the probability simplex $(\Delta^n,g^F,\nabla^{\mathrm{(e)}})$ is doubly totally-umbilical if and only if it is contained in the following mixture family$\mathrm{:}$
\begin{align*}
    \Xi_{a_{i_1},\ldots,a_{i_{\mathcal{A}}},b_{j_1},\ldots,b_{j_{\mathcal{B}}},\phi} = \left\{p\in\Delta^n\mid p(\omega_{i_k}) = a_{i_k}p(\omega_{\phi(i_k)}),\,p(\omega_{j_l})=b_{j_l}\right\},
\end{align*}
where $i_1,\ldots,i_{\mathcal{A}},j_1,\ldots,j_{\mathcal{B}}\in\{1,\ldots,n+1\}$ are $\mathcal{A}+\mathcal{B}=n-m$ distinct integers, $\phi:\{i_1,\ldots,i_{\mathcal{A}}\}\to(\{1,\ldots,n+1\}\setminus\{i_1,\ldots,i_{\mathcal{A}},j_1,\ldots,j_{\mathcal{B}}\})$, and $a_{i_1},\ldots,a_{i_{\mathcal{A}}}$ are positive real numbers and $b_{j_1},\ldots,b_{j_{\mathcal{B}}}\in(0,1)$.\\

In particular, if $n=m+1$, then either one of the following hold$:$
\renewcommand{\labelenumi}{$\mathrm{(\theenumi)}$}
\begin{enumerate}
        \item The submanifold $M$ is doubly totally-umbilical if and only if there exist an\\ $\omega\in\{1,\ldots,m+1\}$ and $0<b<1$ such that $M$ is contained in 
        \begin{align*}
            \{p\in\Delta^{m+1}\mid p(\omega) = b\}.     
        \end{align*} 
        \item The submanifold $M$ is doubly autoparallel if and only if there exist two distinct $\omega_1,\omega_2\in\{1,\ldots,m+1\}$ and an $a>0$ such that $M$ is contained in 
        \begin{align*}
            \{p\in\Delta^{m+1}\mid p(\omega_1) = ap(\omega_2)\}.
        \end{align*}
\end{enumerate}
It is important to note that every doubly autoparallel submanifolds are doubly totally-umbilical submanifolds.\\

We review the submanifold theory of statistical manifolds in Section \ref{dtusec}. 
The proof relies on the fact that the statistical manifold $(\Delta^n,g^F,\nabla^{(\mathrm{e})})$ is a Hessian manifold, namely, a statistical manifold whose affine connection is flat. 
Since Hessian manifolds carry the Hessian curvature tensor, we derive formulas for submanifolds in terms of the Hessian curvature of the ambient space. 
In particular, it turns out that doubly totally-umbilical submanifolds in the probability simplex are also Hessian manifolds.\\

In Theorem \ref{dtudenorm2}, we also classify the doubly totally-umbilical submanifolds in the denormalized state space $(\mathbb{R}^+)^{n+1}$ with the Hessian structure $(g_0,D)$ from \cite{AmariNagaoka2000} or \cite{HF2013}.
This is because the probability simplex $(\Delta^n,g^F,\nabla^{(\mathrm{e})})$ is embedded into $((\mathbb{R}^+)^{n+1},g_0,D)$ as a doubly totally-umbilical statistical submanifold.
Therefore, any doubly totally-umbilical submanifold $M$ of the probability simplex can be embedded into $((\mathbb{R}^+)^{n+1},g_0,D)$ as a doubly totally-umbilical submanifold.
Consequently, the submanifold $M$ can be realized as the intersection of $\Delta^n$ and a doubly totally-umbilical submanifold of the denormalized state space.\\

\section{Preliminaries}\label{sec2}

All objects in this paper are assumed to be smooth, and $M=M^m$ denotes a connected manifold of dimension $m\geq2$.
We denote by $\Gamma(E)$ the set of sections of a vector bundle $E$ over $M$.
In particular, we denote by $T^{(p,q)}M$ the tensor bundle of type $(p,q)$.

\subsection{Statistical manifolds}
Let $g$ be a Riemannian metric on $M$, and $\nabla$ a torsion-free affine connection on $M$.
The triplet $(M,g,\nabla)$ is called a \textit{statistical manifold} if 
\begin{equation}
    (\nabla_X g)(Y,Z) = (\nabla_Y g)(X,Z)
\end{equation}
holds for any vector fields $X,Y,Z\in\Gamma(TM)$.
Here, the pair $(g,\nabla)$ is called the \textit{statistical structure} on $M$.
The Levi-Civita connection of $g$ will be denoted by $\nabla^g$.
The equation $\nabla g = 0$ is equivalent to $\nabla=\nabla^g$.\\

\begin{definition}
    The \textit{conjugate connection} $\nabla^{*}$ of a statistical manifold $(M,g,\nabla)$ is a torsion-free affine connection on $M$ defined by
    \begin{equation}
        Xg(Y,Z)=g(\nabla_X Y,Z)+g(Y,\nabla^{*}_X Z),\quad X,Y,Z\in\Gamma(TM).
    \end{equation} 
\end{definition}

\begin{remark}
    Conjugate connections are often called \textit{dual connections} and are sometimes denoted by $\overline{\nabla}$ instead of $\nabla^{*}$.\\
\end{remark}

\begin{remark}
    For a statistical manifold $(M,g,\nabla)$, the triplet $(M,g,\nabla^{*})$ is also a statistical manifold.
    In fact, the equality $\nabla g=-\nabla^{*}g$ holds. \\
\end{remark}

The statistical manifold in the next example is the main object of this paper and is called the probability simplex.
See \cite{AmariNagaoka2000,8006748,Nay} for its role in information geometry.
For differential geometric aspects of the probability simplex, see \cite{FURUHATA2011S86,HF2013,shima2007geometry,Ohara2017} for details.\\

\begin{example}
    \label{probsimp}
    Let $\Omega_{n+1}=\{1,\ldots,n+1\}$ be a finite set.
    A positive \textit{probability distribution} $p$ on $\Omega_{n+1}$ is a map $p:\Omega_{n+1}\to(0,1)$ such that
    \begin{equation}
        \sum_{\omega=1}^{n+1}p(\omega) = 1.
    \end{equation}
    Denote by $\Delta^n$ the set of all the positive probability distributions on $\Omega_{n+1}$.
    The set $\Delta^n$ is a smooth manifold with an atlas consisting of a single coordinate system $(\eta^1,\ldots,\eta^n)$ defined by $(\eta^1(p),\ldots,\eta^n(p))=(p(1),\ldots,p(n))$.
    The \textit{Fisher metric} $g^F$ and the \textit{exponential connection} $\nabla^{(\mathrm{e})}$ are defined by the following equations$:$
    \begin{align}
        g^F\left(\frac{\partial}{\partial \eta^i},\frac{\partial}{\partial \eta^j}\right) &= \sum_{\omega = 1}^{n+1}\left(\frac{\partial}{\partial \eta^i}\log (p(\omega))\cdot\frac{\partial}{\partial \eta^j}\log (p(\omega))\right)p(\omega)\\
        &= \frac{\delta_{ij}}{\eta^i} + \frac{1}{1-\sum_{l=1}^{n}\eta^l},\label{fishelem}\\[6pt]
        g^F\left(\nabla^{(\mathrm{e})}_{\frac{\partial}{\partial \eta^i}}\frac{\partial}{\partial \eta^j},\frac{\partial}{\partial \eta^k}\right)&=\sum_{\omega = 1}^{n+1}\left(\frac{\partial}{\partial \eta^i}\frac{\partial}{\partial \eta^j}\log (p(\omega))\cdot\frac{\partial}{\partial \eta^k}\log (p(\omega))\right)p(\omega)\\
        &= -\frac{\delta_{ij}\delta_{jk}}{\eta^i} + \frac{1}{\left(1-\sum_{l=1}^{n}\eta^l\right)^2},
    \end{align}
    where $\delta_{ij}$ denotes the Kronecker delta.
    The triplet $(\Delta^n,g^F,\nabla^{(\mathrm{e})})$ is a statistical manifold, known as the \textit{probability simplex}.
    The conjugate connection of $(\Delta^n,g^F,\nabla^{(\mathrm{e})})$ is called the \textit{mixture connection}, often denoted by $\nabla^{(\mathrm{m})}$.\\
\end{example}

\begin{definition}
    \label{alphacon}
    The $\alpha$-\textit{connections} $\nabla^{(\alpha)}$ on a statistical manifold $(M,g,\nabla)$ is a family of affine connections defined by
    \begin{equation}
        \nabla^{(\alpha)}_X Y = \frac{1+\alpha}{2}\nabla_X Y + \frac{1-\alpha}{2}\nabla^{*}_X Y,\quad X,Y\in\Gamma(TM)
    \end{equation}
    for each $\alpha\in\mathbb{R}$.
    By definition, $\nabla^{(1)}=\nabla$ and $\nabla^{(-1)}=\nabla^{*}$ hold.
    Moreover, we have $\nabla^{(-\alpha)}=(\nabla^{(\alpha)})^{*}$.\\
\end{definition}

\newpage

\begin{remark}
    In Definition \ref{alphacon}, the triplet $(M,g,\nabla^{(\alpha)})$ is also a statistical manifold for each $\alpha\in\mathbb{R}$.\\
\end{remark}

The difference tensor $K\in\Gamma(T^{(1,2)}M)$ of a statistical manifold $(M,g,\nabla)$ is defined by 
\begin{align}
    \label{difftendef}
    K_X Y = K(X,Y) &= \nabla_X Y - \nabla^g_X Y,\quad X,Y\in\Gamma(TM).
\end{align}
The conditions $\nabla g = 0$ and $K=0$ are equivalent. 
The difference tensor $K^{(\alpha)}$ of $(M,g,\nabla^{(\alpha)})$ for each $\alpha\in\mathbb{R}$ is given by $K^{(\alpha)}=\alpha K$.\\

\subsection{Curvature tensor fields on the statistical manifold}

For an affine connection $\nabla$ on $M$, we define the \textit{curvature tensor field} $R^\nabla \in \Gamma(T^{(1,3)}M)$ of $\nabla$ by
\begin{equation*}
    R^\nabla (X,Y)Z=\nabla_X \nabla_Y Z-\nabla_Y \nabla_X Z-\nabla_{[X,Y]}Z,
\quad X, Y, Z \in \Gamma(TM). 
\end{equation*}
On a statistical manifold $(M,g,\nabla)$, we often denote $R^{\nabla}$ by $R$, 
$R^{\nabla^g}$ by $R^g$, and $R^{\nabla^{*}}$ by $R^{*}$, and we denote the curvature tensor field of $\nabla^{(\alpha)}$ by $R^{(\alpha)}$. 
For $X,Y,Z,W\in\Gamma(TM)$, the curvature tensor fields are related as follows:
\begin{eqnarray}
    &&
    g(R(X,Y)Z,W) = -g(Z,R^*(X,Y)W), \label{intcurvature}\\
    &&
    R(X,Y)Z = R^g(X,Y)Z + (\nabla^g_X K)(Y,Z) - (\nabla^g_Y K)(X,Z) + \lbrack K_X,K_Y \rbrack Z, \label{Levicurvandstat}\\
    &&
    R(X,Y)Z + R^*(X,Y)Z = 2R^g(X,Y)Z + 2[K_X,K_Y]Z.\label{addcurv}
\end{eqnarray}

\begin{definition}
    A statistical manifold $(M,g,\nabla)$ is said to be \textit{conjugate symmetric} if $R=R^*$ holds.\\
\end{definition}

\begin{remark}
    The conjugate symmetry of $(M,g,\nabla)$ is equivalent to each of the following conditions$:$
    \renewcommand{\labelenumi}{$\operatorname{(\theenumi)}$}
    \begin{enumerate}
        \item $g(R(X,Y)Z,W) = -g(Z,R(X,Y)W)$ for all $X,Y,Z,W\in\Gamma(TM)$.
        \item $\nabla^gK$ is totally symmetric on $M$.\\
    \end{enumerate}
\end{remark}

\begin{definition}
    A statistical manifold $(M,g,\nabla)$ is said to have \textit{constant curvature} $k$ if the following equation holds for some real number $k$$:$
    \begin{equation}
        \label{constcurv}
        R(X,Y)Z=k(g(Y,Z)X-g(X,Z)Y),\quad X,Y,Z\in\Gamma(TM).
    \end{equation}
\end{definition}
If a statistical manifold $(M,g,\nabla)$ has constant curvature, then the statistical manifold $(M,g,\nabla^{*})$ also has constant curvature.
It is easy to see that statistical manifold of constant curvature is also conjugate symmetric.
In fact, the following proposition is known \cite{MR4452143}.\\

\begin{proposition}
    \label{constcon}
    Let $(M,g,\nabla)$ be a statistical manifold. The following conditions are equivalent$:$
    \renewcommand{\labelenumi}{$\operatorname{(\theenumi)}$}
    \begin{enumerate}
        \item The statistical manifold $(M,g,\nabla)$ has constant curvature.
        \item The statistical manifold $(M,g,\nabla)$ is conjugate symmetric and $\nabla$ is projectively flat.\\
    \end{enumerate}
\end{proposition}

If there is a $k^{(\alpha)}\in\mathbb{R}$ such that $R^{(\alpha)}$ satisfies the equation \eqref{constcurv} for each $\alpha\in\mathbb{R}$, then we say that $(M,g,\nabla^{(\alpha)})$ \textit{has constant curvature for each} $\alpha$.
In this case, $k^{(\alpha)}$ can be determined by
\begin{align}
    k^{(\alpha)} = \alpha^2k^{(1)} + (1-\alpha^2)k^{(0)}.
\end{align}
This equation follows from the conjugate symmetry of $(M,g,\nabla)$ and the identity
\begin{align}
    R^{(\alpha)}(X,Y)Z &= R^g(X,Y)Z + \alpha^2[K_X,K_Y]Z\\
    &= \alpha^2R(X,Y)Z + (1-\alpha^2)R^g(X,Y)Z,\quad X,Y,Z\in\Gamma(TM).
\end{align}
Computations of these formulas can be found in many papers, such as \cite{MR4452143,Zhang2007ANO}.\\

\begin{example}
    \label{probsimp2}
    The probability simplex $(\Delta^n,g^F,\nabla^{(\mathrm{e})})$ in Example \ref{probsimp} is a statistical manifold which $(\Delta^n,g^F,\nabla^{(\alpha)})$ is constant curvature for each $\alpha\in\mathbb{R}$.
    In fact, the curvature tensor field of $\nabla^{(\mathrm{e})}$ and $\nabla^{(\mathrm{m})}$ are zero.
    This can be proved since the coordinate system $(\eta^1,\ldots,\eta^n)$ is an affine coordinate system of $\nabla^{(\mathrm{m})}$, that is,
    \begin{align}
        \nabla^{(\mathrm{m})}_{\frac{\partial}{\partial \eta^i}}\frac{\partial}{\partial \eta^j} = 0
    \end{align}
    holds.
    On the other hand, the immersion $\Delta^n\ni p\to(2\sqrt{p(1)},\ldots,2\sqrt{p(n+1)})\in S^{n}(2)$ is an isometric embedding of $(\Delta^n,g^F)$ into the Euclidean sphere $(S^n(2),g_0)$ of radius $2$.
    Thus, $(\Delta^n,g^F)$ can be regarded as an open Riemannian submanifold of $(S^n(2),g_0)$, for details, see \cite{FURUHATA2011S86}.
    Therefore, the Riemannian manifold $(\Delta^n,g^F)$ has constant curvature $\frac{1}{4}$, and 
    \begin{align}
        R^{(\alpha)}(X,Y)Z = \frac{1-\alpha^2}{4}\left(g^F(Y,Z)X-g^F(X,Z)Y\right),\quad X,Y,Z\in\Gamma(T\Delta^n)
    \end{align}
    holds for each $\alpha\in\mathbb{R}$.\\
\end{example}

\subsection{Hessian manifolds and Hessian curvatures}

A statistical manifold $(M,g,\nabla)$ with flat connection $\nabla$ is called a \textit{Hessian manifold}.\\
\begin{definition}
    Let $(M,g,\nabla)$ be a Hessian manifold.
    If there exists a $c\in\mathbb{R}$ such that
    \begin{align}
        (\nabla_X K)(Y,Z) = -\frac{c}{2}\left(g(X,Y)Z+g(X,Z)Y\right),\quad X,Y,Z\in\Gamma(TM),\label{chc}
    \end{align}
    then $(M,g,\nabla)$ is said to have constant Hessian curvature of $c$.
    We abbreviate this by CHC $c$.\\
\end{definition}
As seen in Example \ref{probsimp2}, the probability simplex is a Hessian manifold.
In fact, it has CHC $-1$, see \cite{HF2013}.
The proof of the following proposition can be found in \cite{shima2007geometry} for example.\\

\begin{proposition}
    Let $(M,g,\nabla)$ be a Hessian manifold that has CHC $c$.
    Then the Riemannian manifold $(M,g)$ has constant curvature $-\frac{c}{4}$.\\
\end{proposition}

\begin{example}
    \label{CHC0}
    Let $\mathbb{R}^+=\{y\in\mathbb{R}\mid y>0\}$, denote by $g_0$ the Euclidean metric restricted to $(\mathbb{R}^+)^n$, and let s$D$ be a torsion-free affine connection on $(\mathbb{R}^+)^n$ defined by
    \begin{align*}
        D_{\frac{\partial}{\partial y^i}}\frac{\partial}{\partial y^j} &= -\frac{\delta_{ij}}{y^i}\frac{\partial}{\partial y^i}.
    \end{align*}
    The triplet $((\mathbb{R}^+)^n,g_0,D)$ is a Hessian manifold of CHC $0$ \cite{HF2013}.\\
\end{example}


\section{Doubly totally-umbilical submanifolds of statistical manifolds}
\label{dtusec}
We fix the notation and terminology used throughout this paper, most of which are borrowed from \cite{Furuhata2024Toward}.
Denote by $(\widetilde{M},\widetilde{g},\widetilde{\nabla})$ a statistical manifold of dimension $n=m+p$.
Given an immersion $\iota:M\to\widetilde{M}$, it induces a statistical structure $(g,\nabla)$ on $M$ by the following: 
\begin{align}
    g=\iota^*\widetilde{g},\quad g\left(\nabla_X Y,Z\right)=\widetilde{g}\left(\widetilde{\nabla}_X \iota_*Y,\iota_*Z\right),\quad X,Y,Z\in\Gamma(TM).\label{induce}
\end{align}
Here, the connection on $\iota^{*}T\widetilde{M}$ induced by $\widetilde{\nabla}$ is also denoted by $\widetilde{\nabla}$.\\

\begin{definition}
    Let $(M,g,\nabla),(\widetilde{M},\widetilde{g},\widetilde{\nabla})$ be statistical manifolds.
    If there exists an immersion $\iota:M\to\widetilde{M}$ such that the equation \eqref{induce} holds, then $\iota:(M,g,\nabla)\to(\widetilde{M},\widetilde{g},\widetilde{\nabla})$ is called a \textit{statistical immersion}, and $(M,g,\nabla)$ is called a \textit{statistical submanifold} of $(\widetilde{M},\widetilde{g},\widetilde{\nabla})$.\\
\end{definition}

\begin{example}
    \label{probsimpemb}
    Consider the Hessian manifold $(((\mathbb{R}^+)^{n+1},g_0,D))$ of CHC 0 in Example \ref{CHC0}.
    The probability simplex is embedded into $(\mathbb{R}^+)^{n+1}$ by $\iota:\Delta^n\ni p\to(2\sqrt{p(1)},\ldots,2\sqrt{p(n+1)})\in(\mathbb{R}^+)^{n+1}$.
    Define a coordinate system $(\eta^1,\ldots,\eta^{n+1})$ on $(\mathbb{R}^+)^{n+1}$ by $\eta^i=\frac{(y^i)^2}{4},\,i\in\{1,\ldots,n+1\}$.
    If we denote by $D^*$ the conjugate connection of $((\mathbb{R}^+)^{n+1},g_0,D)$, then $(\eta^1,\ldots,\eta^{n+1})$ is an affine coordinate system of $D^*$, and we have
    \begin{align}
        \label{denorprobsimp}
        g_0\left(\frac{\partial}{\partial \eta^i},\frac{\partial}{\partial \eta^j}\right) &= \frac{\delta_{ij}}{\eta^i},\\
        D_{\frac{\partial}{\partial \eta^i}}\frac{\partial}{\partial \eta^j} =& -\frac{\delta_{ij}}{\eta^i}\frac{\partial}{\partial \eta^i}. 
    \end{align}
    With this coordinate system and the coordinate system $(\eta^1,\ldots,\eta^n)$ on $\Delta^n$ defined in Example \ref{probsimp}, the embedding $\iota:\Delta^n\to(\mathbb{R}^+)^{n+1}$ is expressed by $\iota(\eta^1,\ldots,\eta^n)=(\eta^1,\ldots,\eta^n,1-\sum_{i=1}^n\eta^i)$.
    It is easy to see that the statistical structure $(g^F,\nabla^{(\mathrm{e})})$ on $\Delta^n$ is the one induced by $(g_0,D)$, thus the probability simplex $(\Delta^n,g^F,\nabla^{(\mathrm{e})})$ is a statistical submanifold of $((\mathbb{R}^+)^{n+1},g_0,D)$.
    Following \cite{AmariNagaoka2000, Fujiwara2026Companion}, the statistical manifold $((\mathbb{R}^+)^{n+1},g_0,D)$ is called the \textit{denormalized state space}.\\
\end{example}

\begin{example}
    For an embedding $\iota:\Delta^m\to\Delta^n$ where $m\leq n$, suppose there is a family of non-empty subsets $\{C_1,\ldots,C_{m+1}\}\subset\Omega_{n+1}$ such that 
    \begin{align*}
        \Omega_{m+1} = \bigsqcup_{l = 1}^{m+1} C_l
    \end{align*}
    is a disjoint union.
    The embedding $\iota$ is called a \textit{Markov embedding} if there exist functions $Q_l:\Omega_{n+1}\to[0,\infty)$ whose support is contained in $C_l$ for each $l\in\Omega_{m+1}$ such that
    \begin{align}
        \label{markovemb}
        \iota(p) = \sum_{l = 1}^{m+1}p(l)Q_l,\quad p\in\Delta^{m}.
    \end{align}
    With the statistical structure $(g^F,\nabla^{(\mathrm{e})})$ on the probability simplex in Example \ref{probsimp}, Markov embeddings $\iota:(\Delta^m,g^F,\nabla^{(\mathrm{e})})\to(\Delta^n,g^F,\nabla^{(\mathrm{e})})$ are statistical immersions.
    In fact, it is known that the scalar multiples of the Fisher metric, the $\alpha$-connections of $(g^F,\nabla^{\mathrm{(e)}})$ are the only $(0,2)$-type tensor field, affine connections, respectively, such that it preserves any Markov embeddings \cite{cencov2000statistical}.\\
\end{example}

\begin{definition}
    If a statistical immersion $f:(M,g,\nabla)\to(\widetilde{M},\widetilde{g},\widetilde{\nabla})$ is a diffeomorphism, then $f$ is called a \textit{statistical diffeomorphism}.\\
\end{definition}

Given a statistical immersion $\iota:(M,g,\nabla)\to(\widetilde{M},\widetilde{g},\widetilde{\nabla})$, we decompose the vector bundle $\iota^*T\widetilde{M}$ with respect to $\widetilde{g}$ by
\begin{align}
    \iota^*T\widetilde{M} = \iota_*TM\oplus TM^{\perp}.
\end{align}
From this decomposition, we define $B\in\Gamma(TM^{\perp}\otimes T^{(0,2)}M)$, $A\in\Gamma(T^{(1,1)}M\otimes(TM^{\perp})^*)$, and a connection $\nabla^{\perp}$ on $TM^{\perp}$ by the following formulas:
\begin{align}
    \widetilde{\nabla}_X&\iota_{*} Y = \iota_*(\nabla_XY) + B(X,Y),\label{Gauss}\\
    \widetilde{\nabla}_X\xi &= -\iota_*A_{\xi}X + \nabla^{\perp}_X\xi,\quad X,Y\in\Gamma(TM),\,\xi\in\Gamma(TM^{\perp}).\label{Wein}
\end{align}
We call $B$ the second fundamental form, $A$ the shape operator, and $\nabla^{\perp}$ the normal connection of $\iota$, all with respect to $\widetilde{\nabla}$.
For each $\alpha\in\mathbb{R}$, it is easy to see that $(M,g,\nabla^{(\alpha)})$ is a statistical submanifold of $(\widetilde{M},\widetilde{g},\widetilde{\nabla}^{(\alpha)})$ since the equation \eqref{induce} holds.
Thus, we define $B^{(\alpha)}$, $A^{(\alpha)}$, and $\nabla^{\perp(\alpha)}$ for $\iota$ by \eqref{Gauss} and \eqref{Wein} with respect to $\widetilde{\nabla}^{(\alpha)}$ for each $\alpha\in\mathbb{R}$.\\

\begin{remark}
    Let $\iota:(M,g,\nabla)\to(\widetilde{M},\widetilde{g},\widetilde{\nabla})$ be a statistical immersion.
    For each $\alpha\in\mathbb{R}$, the following equations hold$:$
    \begin{align}
        B^{(0)} &= \frac{B^{(\alpha)}+B^{(-\alpha)}}{2},\\
        A^{(0)} &= \frac{A^{(\alpha)}+A^{(-\alpha)}}{2}.
    \end{align}
\end{remark}

The following Propositions \ref{basic1} and \ref{basic2} are obtained by simple computations (see \cite{FuruhataHasegawa2016} for example).\\
\begin{proposition}
    \label{basic1}
    Let $\iota:(M,g,\nabla)\to(\widetilde{M},\widetilde{g},\widetilde{\nabla})$ be a statistical immersion.
    For each $\alpha\in\mathbb{R}$, we have
    \begin{align}
        \widetilde{g}(B^{(\alpha)}(X,Y),\xi) &= g(A^{(-\alpha)}_{\xi}X,Y),\\
        X\widetilde{g}(\xi,\eta) = \widetilde{g}(\nabla^{\perp(\alpha)}_X\xi,&\eta) + \widetilde{g}(\xi,\nabla^{\perp(-\alpha)}_X\eta),
    \end{align}
    where $X,Y\in\Gamma(TM)$ and $\xi,\eta\in\Gamma(TM^{\perp})$.\\
\end{proposition}

\begin{proposition}
    \label{basic2}
    Let $\iota:(M,g,\nabla)\to(\widetilde{M},\widetilde{g},\widetilde{\nabla})$ be a statistical immersion, and $\widetilde{R},R$ be the curvature tensor fields of $\widetilde{\nabla},\nabla$, respectively.
    The following equations hold for $X,Y,Z,W\in\Gamma(TM)$ and $\xi,\eta\in\Gamma(TM^{\perp})$$:$
    \begin{align}
        \widetilde{g}\left(\widetilde{R}(\iota_*X,\iota_*Y)\iota_*Z,\iota_*W\right) &= g\Big(R(X,Y)Z - A_{B(Y,Z)}X + A_{B(X,Z)}Y,W\Big),\\
        \widetilde{g}\left(\widetilde{R}(\iota_*X,\iota_*Y)\iota_*Z,\xi\right) &= \widetilde{g}\Big(\left(\nabla_X B\right)(Y,Z) - \left(\nabla_Y B\right)(X,Z),\xi\Big),\\
        \widetilde{g}\left(\widetilde{R}(\iota_*X,\iota_*Y)\xi,\iota_*Z\right) &= g\Big(\left(\nabla_Y A\right)(X,\xi) - \left(\nabla_X A\right)(\xi,Y),Z\Big),\\
        \widetilde{g}\left(\widetilde{R}(\iota_*X,\iota_*Y)\xi,\eta\right) &= \widetilde{g}\Big(R^{\nabla^{\perp}}(X,Y)\xi-B(X,A_{\xi}Y)+B(Y,A_{\xi}X),\eta\Big).    \end{align}
    Here, 
    \begin{align}
        (\nabla_X B)(Y,Z) &= \nabla^{\perp}_XB(Y,Z) - B(\nabla_X Y,Z) - B(Y,\nabla_X Z),\\
        (\nabla_X A)(\xi,Y) &= \nabla_X A_{\xi}Y - A_{\nabla^{\perp}_X\xi}Y - A_{\xi}\nabla_X Y,
    \end{align}
    where $X,Y\in\Gamma(TM)$ and $\xi\in\Gamma(TM^{\perp})$, and $R^{\nabla^{\perp}}$ is the curvature tensor field of $\nabla^{\perp}$.\\
\end{proposition}

The following important classes of statistical submanifolds were introduced in \cite{Furuhata2024Toward, Satoh2020StatisticalSubmanifolds}.
For a statistical immersion $\iota:(M,g,\nabla)\to(\widetilde{M},\widetilde{g},\widetilde{\nabla})$, we define the \textit{mean curvature tensor field} $H^{(\alpha)}\in\Gamma(TM^{\perp})$ with respect to $\widetilde{\nabla}^{(\alpha)}$ for each $\alpha\in\mathbb{R}$ by
\begin{align}
    H^{(\alpha)} = \frac{1}{m}\sum_{i=1}^mB^{(\alpha)}(e_i,e_i),
\end{align}
where $\{e_1,\ldots,e_m\}$ is an orthonormal frame of $(M,g)$.\\

\begin{definition}
    \label{defum}
    Let $\iota:(M,g,\nabla)\to(\widetilde{M},\widetilde{g},\widetilde{\nabla})$ be a statistical submanifold.
    \renewcommand{\labelenumi}{$\operatorname{(\theenumi)}$}
    \begin{enumerate}
        \item The immersed submanifold $M$ is said to be \textit{doubly autoparallel} if $B^{(1)}=B^{(-1)}=0$.\\
        \item The immersed submanifold $M$ is said to be \textit{doubly totally-umbilical} if $B^{(1)}=H^{(1)}\otimes g$ and $B^{(-1)}=H^{(-1)}\otimes g$ hold.\\
    \end{enumerate}
\end{definition}

\begin{remark}[\cite{Furuhata2024Toward}]
    \label{umrem}
    The following hold in Definition \ref{defum}.
    \renewcommand{\labelenumi}{$\operatorname{(\theenumi)}$}
    \begin{enumerate}
        \item If there is an $\alpha,\beta\in\mathbb{R}$ such that $\alpha\neq\beta$ and $B^{(\alpha)}=B^{(\beta)}=0$ hold, then the submanifold $M$ is doubly autoparallel.
        \item Similarly, if there is an $\alpha,\beta\in\mathbb{R}$ such that $\alpha\neq\beta$, $B^{(\alpha)}=H^{(\alpha)}\otimes g$ and $B^{(\beta)}=H^{(\beta)}\otimes g$ are satisfied, then the submanifold $M$ is doubly totally-umbilical.
        \item For each $\alpha\in\mathbb{R}$, the condition $B^{(\alpha)}=H^{(\alpha)}\otimes g$ is equivalent to 
        \begin{align}
            \label{umcon}
            A^{(-\alpha)}_{\xi}X = \widetilde{g}(H^{(\alpha)},\xi)X,\quad X\in\Gamma(TM),\,\xi\in\Gamma(TM^{\perp}).
        \end{align}
    \end{enumerate}
\end{remark}

\begin{remark}
    If the submanifold $M$ of $(\widetilde{M},\widetilde{g},\widetilde{\nabla})$ satisfies $B^{(\alpha)}=0$ for some $\alpha\in\mathbb{R}$, then $M$ is called a $\widetilde{\nabla}^{(\alpha)}$-\textit{autoparallel} submanifold.
    It is known that $M$ is a $\nabla^{(\mathrm{e})}$-autoparallel submanifold of $(\Delta^n,g^F,\nabla^{(\mathrm{e})})$ if and only if $M$ is an \textit{exponential family}, and $M$ is a $\nabla^{(\mathrm{m})}$-autoparallel submanifold of $(\Delta^n,g^F,\nabla^{(\mathrm{e})})$ if and only if $M$ is a \textit{mixture family} \cite{MR1800071}.\\
\end{remark}

\begin{example}
    \label{probsimpisdtu}
    The embedding $\iota:\Delta^n\to(\mathbb{R}^+)^{n+1}$ described in Example \ref{probsimpemb} shows that $\Delta^n$ is a doubly totally-umbilical submanifold of $((\mathbb{R}^+)^{n+1},g_0,D)$.
    See \cite{FURUHATA2011S86} for details.\\
\end{example}

\subsection{New results on doubly totally-umbilical submanifolds}
Even when a statistical immersion exists, properties of the ambient space such as conjugate symmetry and having constant curvature are not necessarily inherited by the submanifold.
However, if the submanifold is doubly totally-umbilical, these desirable properties are inherited.
We begin by computing the equations in Proposition~\ref{basic2} for doubly totally-umbilical submanifolds.\\

\begin{lemma}
    Let $\iota:(M,g,\nabla)\to(\widetilde{M},\widetilde{g},\widetilde{\nabla})$ be a statistical immersion and $\widetilde{R}^{(\alpha)},R^{(\alpha)}$ be the curvature tensor fields of $\widetilde{\nabla}^{(\alpha)},\nabla^{(\alpha)}$, respectively.
    If $M$ is doubly totally-umbilical, then the following equations hold for $X,Y,Z,W\in\Gamma(TM)$ and $\xi,\eta\in\Gamma(TM^{\perp})$$:$
    \begin{align}
        \widetilde{g}\left(\widetilde{R}^{(\alpha)}(\iota_*X,\iota_*Y)\iota_*Z,\iota_*W\right) =& g\Big(R^{(\alpha)}(X,Y)Z,W\Big)\\
        &- \widetilde{g}\left(H^{(\alpha)},H^{(-\alpha)}\right)\Big(g(Y,Z)g(X,W)-g(X,Z)g(Y,W)\Big),\label{umbeq1}\\
        \widetilde{g}\left(\widetilde{R}^{(\alpha)}(\iota_*X,\iota_*Y)\iota_*Z,\xi\right) &= \widetilde{g}\Big(g(Y,Z)\nabla^{\perp{(\alpha)}}_X H^{(\alpha)} - g(X,Z)\nabla^{\perp{(\alpha)}}_Y H^{(\alpha)},\xi\Big),\label{umbeq2}\\
        \widetilde{g}\left(\widetilde{R}^{(\alpha)}(\iota_*X,\iota_*Y)\xi,\iota_*Z\right) &= -\widetilde{g}\Big(g(Y,Z)\nabla^{\perp{(-\alpha)}}_X H^{(-\alpha)} - g(X,Z)\nabla^{\perp{(-\alpha)}}_Y H^{(-\alpha)},\xi\Big),\label{umbeq3}\\
        \widetilde{g}\left(\widetilde{R}^{(\alpha)}(\iota_*X,\iota_*Y)\xi,\eta\right) =& \widetilde{g}\Big(R^{\nabla^{\perp(\alpha)}}(X,Y)\xi,\eta\Big).\label{umbeq4}
    \end{align}
\end{lemma}
\begin{proof}
    We prove the equations using Proposition~\ref{basic2} one by one.
    For each $X,Y,Z,W\in\Gamma(TM)$, it holds that
    \begin{align}
        \widetilde{g}\left(\widetilde{R}^{(\alpha)}(\iota_*X,\iota_*Y)\iota_*Z,\iota_*W\right) =& g\Big(R^{(\alpha)}(X,Y)Z - g(Y,Z)A^{(\alpha)}_{H^{(\alpha)}}X + g(X,Z)A^{(\alpha)}_{H^{(\alpha)}}Y,W\Big)\\
        =& g\Big(R^{\alpha}(X,Y)Z,W\Big)\\
        &- \widetilde{g}\left(H^{(\alpha)},H^{(-\alpha)}\right)\Big(g(Y,Z)g(X,W)-g(X,Z)g(Y,W)\Big),
    \end{align}
    where we used \eqref{umcon} in the last equality.
    To prove equations \eqref{umbeq2} and \eqref{umbeq3}, it suffices to compute $\nabla^{(\alpha)} B^{(\alpha)}$ and $\nabla^{(\alpha)} A^{(\alpha)}$.
    For each $X,Y,Z\in\Gamma(TM)$ and $\xi\in\Gamma(TM^{\perp})$, we have
    \begin{align}
        \left(\nabla^{(\alpha)}_X B^{(\alpha)}\right)(Y,Z) =& \nabla^{\perp{(\alpha)}}_XB^{(\alpha)}(Y,Z) - B^{(\alpha)}\left(\nabla^{(\alpha)}_X Y,Z\right) - B^{(\alpha)}\left(Y,\nabla^{(\alpha)}_Y Z\right),\\
        =&Xg(Y,Z)H^{(\alpha)} + g(Y,Z)\nabla^{\perp{(\alpha)}}H^{(\alpha)}\\
        &-g\left(\nabla^{(\alpha)}_X Y,Z\right)H^{(\alpha)} - g\left(Y,\nabla^{(\alpha)}_X Z\right)H^{(\alpha)}\\
        =&\left(\nabla^{(\alpha)}_X g\right)(Y,Z)H^{(\alpha)} + g(Y,Z)\nabla^{\perp(\alpha)}_X H^{(\alpha)},
    \end{align}
    and from \eqref{umcon}, we have
    \begin{align}
        (\nabla^{(\alpha)}_X A^{(\alpha)})(\xi,Y) &= \nabla^{(\alpha)}_X A^{(\alpha)}_{\xi}Y - A^{(\alpha)}_{\nabla^{\perp^{(\alpha)}}_X\xi}Y - A^{(\alpha)}_{\xi}\nabla^{(\alpha)}_X Y\\
        &= \left(\nabla^{(\alpha)}_X\widetilde{g}\left(H^{-(\alpha)},\xi\right)Y\right) - \widetilde{g}\left(H^{(-\alpha)},\nabla^{\perp(\alpha)}_X \xi\right)Y - \widetilde{g}\left(H^{(-\alpha)},\xi\right)\nabla^{(\alpha)}_X Y\\
        &=X\widetilde{g}\left(H^{(-\alpha)},\xi\right)Y - \widetilde{g}\left(H^{(-\alpha)},\nabla^{\perp(\alpha)}_X \xi\right)Y\\
        &=\widetilde{g}\left(\nabla^{\perp(\alpha)}_X H^{(-\alpha)},\xi\right)Y.
    \end{align}
    Equation \eqref{umbeq2} is obtained since $\nabla g$ is symmetric.
    Lastly, the equation \eqref{umbeq4} is obtained by \eqref{umcon} and the symmetry of $B^{(\alpha)}$.
\end{proof}

\begin{proposition}
    \label{conjandcon}
    Let $\iota:(M,g,\nabla)\to(\widetilde{M},\widetilde{g},\widetilde{\nabla})$ be a statistical immersion, where $M$ is a doubly totally-umbilical submanifold.
    \renewcommand{\labelenumi}{$\operatorname{(\theenumi)}$}
    \begin{enumerate}
        \item If $(\widetilde{M},\widetilde{g},\widetilde{\nabla})$ is conjugate symmetric, then so is $(M,g,\nabla)$.
        \item If $(\widetilde{M},\widetilde{g},\widetilde{\nabla})$ has constant curvature $\widetilde{k}$, then $\widetilde{g}\left(H^{(1)},H^{(-1)}\right)$ is a constant function and $(M,g,\nabla)$ has constant curvature of $k=\widetilde{k}+\widetilde{g}\left(H^{(1)},H^{(-1)}\right)$.\\
    \end{enumerate}
\end{proposition}

\begin{proof}
    Claim $(1)$ immediately follows from equation \eqref{umbeq1}.
    Suppose that $(\widetilde{M},\widetilde{g},\widetilde{\nabla})$ has constant curvature $\widetilde{k}$.
    We prove that 
    \begin{align}
        \label{parmean}
        \nabla^{\perp(1)}_X H^{(1)}=0,\quad X\in\Gamma(TM)
    \end{align}
    holds.
    The statistical manifold $(\widetilde{M},\widetilde{g},\widetilde{\nabla})$ is also conjugate symmetric from Proposition~\ref{constcon}.
    For any orthonormal pair $\{X,Y\}$ on $(M,g)$ and $\xi\in\Gamma(TM^{\perp})$, from \eqref{umbeq2} we have
    \begin{align}
        \widetilde{g}\left(\widetilde{R}^{(1)}(\iota_*X,\iota_*Y)\iota_*Y,\xi\right) = \widetilde{g}\Big(\nabla^{\perp{(1)}}_X H^{(1)},\xi\Big).
    \end{align}
    Since the left-hand side of \eqref{umbeq2} vanishes if $(\widetilde{M},\widetilde{g},\widetilde{\nabla})$ has constant curvature, we obtain \eqref{parmean}.
    The equation $\nabla^{\perp(-1)}_X H^{(-1)}=0$ can be proved in the same manner.
    Thus, for any $X\in\Gamma(TM)$ we have
    \begin{align}
        X\widetilde{g}\left(H^{(1)},H^{(-1)}\right) = \widetilde{g}\left(\nabla^{\perp(1)}_X H^{(1)},H^{(-1)}\right) + \widetilde{g}\left(H^{(1)},\nabla^{\perp(-1)}_X H^{(-1)}\right) = 0,
    \end{align}
    so we conclude that $\widetilde{g}\left(H^{(1)},H^{(-1)}\right)$ is a constant function.
    For any orthonormal pair $\{X,Y\}$ on $(M,g)$, from \eqref{umbeq1} we have
    \begin{align}
        g\Big(R^{(1)}(X,Y)Y,X\Big) &= \widetilde{g}\left(\widetilde{R}^{(1)}(\iota_*X,\iota_*Y)\iota_*Y,\iota_*X\right) + \widetilde{g}\left(H^{(1)},H^{(-1)}\right)\\
        &=\widetilde{k} + \widetilde{g}\left(H^{(1)},H^{(-1)}\right),
    \end{align}
    therefore, the statistical manifold $(M,g,\nabla)$ has constant curvature.\\
\end{proof}

\begin{remark}
    Let $\iota:(M,g,\nabla)\to(\widetilde{M},\widetilde{g},\widetilde{\nabla})$ be a statistical immersion, where $M$ is a doubly totally-umbilical submanifold and $(\widetilde{M},\widetilde{g},\widetilde{\nabla})$ has constant curvature.
    In \cite{amami}, it is stated that if the function $\widetilde{g}\left(H^{(1)},H^{(-1)}\right)$ is constant, then $(M,g,\nabla)$ also has constant curvature, however, by Proposition~\ref{conjandcon} we see that this function is always constant under these assumptions.\\
\end{remark}

\begin{example}
    \label{umprob}
    Fix $\omega\in\{1,\ldots,n+1\}$ and $b\in(0,1)$.
    Then
    \begin{align}
        M = \{p\in\Delta^n\mid p(\omega)=b\}
    \end{align}
    is a doubly totally-umbilical submanifold of the probability simplex $(\Delta^n,g^F,\nabla^{(\mathrm{e})})$.
    Indeed, we have $B^{(0)} = H^{(0)}\otimes g$, since the image of $\Delta^n\ni p\to(2\sqrt{p(1)},\ldots,2\sqrt{p(n+1)})\in S^{n}(2)$ restricted to $M$ is an open Riemannian submanifold of $S^{n-1}(2\sqrt{1-b})\subset S^n(2)$, which is a well-known totally-umbilical submanifold in Riemannian geometry (see \cite{Chen2019} for example).
    In order to prove that $M$ is a doubly totally-umbilical submanifold, we prove $B^{(-1)}=0$.
    Consider the global affine coordinate system $(\eta^1,\ldots,\eta^n)$ of $\nabla^{(\mathrm{m})}$ defined in Example \ref{probsimp}.
    Since $M$ is a hyperplane of $\Delta^n$ with respect to the coordinate system $(\eta^1,\ldots,\eta^n)$, it is clear that $M$ is a $\nabla^{(\mathrm{m})}$-autoparallel submanifold, thus we have $B^{(-1)}=0$.\\
\end{example}

\begin{proposition}
    \label{nontrib}
    Let $\iota:(M,g,\nabla)\to(\widetilde{M},\widetilde{g},\widetilde{\nabla})$ be a statistical immersion, where $M$ is a doubly totally-umbilical submanifold.
    If the following inequality holds for each orthonormal pair of tangent vectors $\{X,Y\}$ on $(\widetilde{M},\widetilde{g})$, then the difference tensor field $K$ of $(M,g,\nabla)$ does not vanish$\mathrm{:}$
    \begin{equation}
        \label{ineq1}
        \widetilde{g}\left(\widetilde{K}_X\widetilde{K}_YY-\widetilde{K}_Y\widetilde{K}_XY,X\right)<0.
    \end{equation}
    Here, $\widetilde{K}$ is the difference tensor of $(\widetilde{g},\widetilde{\nabla})$.
\end{proposition}

\begin{proof}
    Assume that $K=0$ holds.
    We have 
    \begin{align*}
        \widetilde{K}_{\iota_*X}{\iota_*Y} = g(X,Y)\left(H^{(1)}-H^{(0)}\right)
    \end{align*}
    for each $X,Y\in\Gamma(TM)$.
    Thus, for each orthonormal pair $\{X,Y\}$ on $(M,g)$, we have
    \begin{align*}
        \widetilde{g}\left(\widetilde{K}_{\iota_*X}\widetilde{K}_{\iota_*Y}\iota_*Y-\widetilde{K}_{\iota_*Y}\widetilde{K}_{\iota_*X}\iota_*Y,\iota_*X\right)&=\widetilde{g}\left(\widetilde{K}_{\iota_*X}\iota_*X,\widetilde{K}_{\iota_*Y}\iota_*Y\right)\\
        &=\|H^{(1)}-H^{(0)}\|_{\widetilde{g}}^2\geq 0,
    \end{align*}
    which contradicts the inequality \eqref{ineq1}.\\
\end{proof}

\subsection{Doubly totally-umbilical hypersurfaces}
Let $\iota:(M^m,g,\nabla)\to(\widetilde{M}^{m+1},\widetilde{g},\widetilde{\nabla})$ be a statistical immersion. 
If $n=m+1$, the statistical submanifold $M$ is called a \textit{statistical hypersurface} of $\widetilde{M}$.
We also assume that $M$ is orientable in this subsection.\\

Since $M$ is orientable, let $\boldsymbol{n}$ be a unit normal vector field along $\iota$.
For each $\alpha\in\mathbb{R}$, there exists a symmetric $h^{(\alpha)}\in\Gamma(T^{(0,2)}M)$ such that $B^{(\alpha)} = \boldsymbol{n}\otimes h^{(\alpha)}$.
There also exists an 1-form $\tau^{(\alpha)}\in\Gamma(T^*M)$ for each $\alpha\in\mathbb{R}$ such that
\begin{align}
    \nabla^{\perp{(\alpha)}}_X \boldsymbol{n} = \tau^{(\alpha)}(X)\boldsymbol{n},\quad X\in\Gamma(TM).
\end{align}
For each $\alpha\in\mathbb{R}$, the relations
\begin{align}
    h^{(0)} =& \frac{h^{(\alpha)}+h^{(-\alpha)}}{2},\\
    \tau^{(\alpha)} &= -\tau^{(-\alpha)}
\end{align}
hold, since $\tau^{(0)} = 0$.\\

\begin{proposition}
    \label{hypsur}
    Let $\iota:(M^m,g,\nabla)\to(\widetilde{M}^{m+1},\widetilde{g},\widetilde{\nabla})$ be statistical hypersurface immersion, where $M$ is a doubly totally-umbilical submanifold.
    For each $\alpha\in\mathbb{R}$, we have
    \begin{align}
        h^{(\alpha)} &= \frac{1}{m}(\mathrm{tr}_g h^{(\alpha)})g,\\
        H^{(\alpha)} &= \frac{1}{m}\mathrm{tr}_g h^{(\alpha)}\boldsymbol{n},\\
        A^{(\alpha)}_{\boldsymbol{n}}X &= \frac{1}{m}\mathrm{tr}_g h^{(-\alpha)}X,\quad X\in\Gamma(TM).
    \end{align}
\end{proposition}
\begin{proof}
    The proof is obtained by a straightforward computation using \eqref{umcon}.\\
\end{proof}

\begin{proposition}
    Let $\iota:(M^m,g,\nabla)\to(\widetilde{M}^{m+1},\widetilde{g},\widetilde{\nabla})$ be statistical hypersurface immersion, where $M$ is a doubly totally-umbilical submanifold.
    If $(\widetilde{M},\widetilde{g},\widetilde{\nabla})$ is conjugate symmetric, then for each $\alpha\in\mathbb{R}$, we have
    \begin{align}
        \label{umhyp1}
        d(\mathrm{tr}_g h^{(\alpha)} - \mathrm{tr}_g h^{(-\alpha)}) = -\left(\mathrm{tr}_g h^{(\alpha)}\tau^{(\alpha)}-\mathrm{tr}_g h^{(-\alpha)}\tau^{(-\alpha)}\right).
    \end{align}
    Moreover, if $(\widetilde{M},\widetilde{g},\widetilde{\nabla})$ has constant curvature, then
    \begin{align}
        \label{umhyp2}
        d(\mathrm{tr}_gh^{(1)}) = -\mathrm{tr}_gh^{(1)}\tau^{(1)}.
    \end{align}
\end{proposition}

\begin{proof}
    If $(\widetilde{M},\widetilde{g},\widetilde{\nabla})$ is conjugate symmetric, then $\widetilde{R}^{(-\alpha)}=\widetilde{R}^{(\alpha)}$.
    Hence, by \eqref{umbeq2} and \eqref{umbeq3}, we obtain
    \begin{align}
        \label{dermean}
        \nabla^{\perp(\alpha)}_X H^{(\alpha)} = \nabla^{\perp(-\alpha)}_X H^{(-\alpha)} 
    \end{align}
    for each $X\in\Gamma(TM)$.
    Since
    \begin{align}
        \nabla^{\perp(\alpha)}_X H^{(\alpha)} = \frac{1}{m}X\mathrm{tr}_g h^{(\alpha)}\boldsymbol{n} + \frac{1}{m}\mathrm{tr}_g h^{(\alpha)}\tau^{(\alpha)}(X)\boldsymbol{n}
    \end{align}
    for each $X\in\Gamma(TM)$, comparing the coefficients of $\boldsymbol{n}$ yields \eqref{umhyp1}.
    If $(\widetilde{M},\widetilde{g},\widetilde{\nabla})$ has constant curvature, then \eqref{parmean} holds.
    Thus, \eqref{umhyp2} follows from the same computation with $\alpha=1$.\\
\end{proof}

\begin{proposition}
    \label{umhyp3}
    Let $\iota:(M^m,g,\nabla)\to(\widetilde{M}^{m+1},\widetilde{g},\widetilde{\nabla})$ be statistical hypersurface immersion, where $M$ is a doubly totally-umbilical submanifold.
    If $(\widetilde{M},\widetilde{g},\widetilde{\nabla}^{(\alpha)})$ has constant curvature for each $\alpha\in\mathbb{R}$, then $\mathrm{tr}_gh^{(\alpha)}$ is a constant function for each $\alpha\in\mathbb{R}$.\\
\end{proposition}
\begin{proof}
    By Proposition~\ref{conjandcon} $(2)$, the function $\widetilde{g}(H^{(0)},H^{(0)})$ is constant, hence $\mathrm{tr}_gh^{(0)}$ is also constant.
    Since $\widetilde{g}(H^{(1)},H^{(-1)})$ is also constant, we have
    \begin{align}
        0 = m^2X\widetilde{g}(H^{(1)},H^{(-1)}) &= X\left(\mathrm{tr}_gh^{(1)}\cdot\mathrm{tr}_gh^{(-1)}\right)\\
        &= X\left(\mathrm{tr}_gh^{(1)}(2\mathrm{tr}_gh^{(0)}-\mathrm{tr}_gh^{(1)})\right)\\
        &=2\left(\mathrm{tr}_gh^{(0)}-\mathrm{tr}_gh^{(1)}\right)X\mathrm{tr}_gh^{(1)}.
    \end{align}
    Since the manifiold $M$ is connected, the function $\mathrm{tr}_gh^{(1)}$ is constant, and consequently so is $\mathrm{tr}_gh^{(-1)}$.
    The function $\mathrm{tr}_gh^{(\alpha)}$ is also constant for each $\alpha\in\mathbb{R}$, since 
    \begin{align}
        \mathrm{tr}_gh^{(\alpha)} = \frac{1+\alpha}{2}\mathrm{tr}_gh^{(1)} + \frac{1-\alpha}{2}\mathrm{tr}_gh^{(-1)}
    \end{align}
    holds.\\
\end{proof}

The following corollary is obtained immediately from equation \eqref{umhyp2} and Proposition~\ref{umhyp3}.\\

\begin{cor}
    \label{umhyp4}
    Let $\iota:(M^m,g,\nabla)\to(\widetilde{M}^{m+1},\widetilde{g},\widetilde{\nabla})$ be a statistical hypersurface immersion, where $M$ is a doubly totally-umbilical submanifold.
    If $(\widetilde{M},\widetilde{g},\widetilde{\nabla}^{(\alpha)})$ has constant curvature for each $\alpha\in\mathbb{R}$ and $M$ is not doubly autoparallel, then $\tau^{(\alpha)}=0$ holds for each $\alpha\in\mathbb{R}$.\\
\end{cor}

Laslty, we obtain the following theorem.

\begin{theorem}
    Let $\iota:(M^m,g,\nabla)\to(\widetilde{M}^{m+1},\widetilde{g},\widetilde{\nabla})$ be a statistical hypersurface immersion, where $M$ is a doubly totally-umbilical submanifold.
    If $(\widetilde{M},\widetilde{g},\widetilde{\nabla}^{(\alpha)})$ has constant curvature for each $\alpha\in\mathbb{R}$ and $\widetilde{k}^{(1)}\neq\widetilde{k}^{(0)}$, then there exists a $\beta\in\mathbb{R}$ such that $M$ is a $\nabla^{(\beta)}$-autoparallel submanifold of $\widetilde{M}$.\\
\end{theorem}

\begin{proof}
    If $M$ is doubly autoparallel, the statement is clear.
    Suppose that $M$ is not doubly autoparallel.
    We first show that $\mathrm{tr}_gh^{(1)}\neq\mathrm{tr}_gh^{(0)}$ holds, where they are constant functions by Proposition~\ref{umhyp3}.
    If we assume that $\mathrm{tr}_gh^{(1)}=\mathrm{tr}_gh^{(0)}$ holds, then by Corollary~\ref{umhyp4}, for any $X\in\Gamma(TM)$ we have
    \begin{align}
        \widetilde{K}_{\iota_*X} \boldsymbol{n} = 0.
    \end{align}
    Here,
    \begin{align}
        \widetilde{K}_{\boldsymbol{n}}\boldsymbol{n} &= \sum_{i=1}^m\widetilde{g}\left(\widetilde{K}_{\boldsymbol{n}}\boldsymbol{n},\iota_*e_i\right)\iota_*e_i + \widetilde{g}\left(\widetilde{K}_{\boldsymbol{n}}\boldsymbol{n},\boldsymbol{n}\right)\boldsymbol{n}\\
        &= \widetilde{g}\left(\widetilde{K}_{\boldsymbol{n}}\boldsymbol{n},\boldsymbol{n}\right)\boldsymbol{n}
    \end{align}
    where $\{e_1,\ldots,e_m\}$ is an orthonormal frame of $(M,g)$.
    Consequently,
    \begin{align}
        0 &= \widetilde{g}\left(\left[\widetilde{K}_{\iota_*X},\widetilde{K}_{\boldsymbol{n}}\right]\boldsymbol{n},\iota_*X\right)\\
        &= \widetilde{g}\left(\widetilde{R}^{(1)}(\iota_*X,\boldsymbol{n})\boldsymbol{n} - \widetilde{R}^{(0)}(\iota_*X,\boldsymbol{n})\boldsymbol{n},\iota_*X\right)\\
        &=\widetilde{k}^{(1)}-\widetilde{k}^{(0)},
    \end{align}
    where $X$ is a unit vector on $(M,g)$.
    This contradicts $\widetilde{k}^{(1)}\neq\widetilde{k}^{(0)}$.
    Therefore, since $\mathrm{tr}_gh^{(1)}\neq\mathrm{tr}_gh^{(0)}$ holds, we define $\beta\in\mathbb{R}$ by 
    \begin{align}
        \beta = \frac{\mathrm{tr}_gh^{(0)}}{\mathrm{tr}_gh^{(0)}-\mathrm{tr}_gh^{(1)}}.
    \end{align}
    Then we obtain
    \begin{align}
        \mathrm{tr}_g h^{(\beta)} &= \frac{1+\beta}{2}\mathrm{tr}_g h^{(1)} + \frac{1-\beta}{2}\mathrm{tr}_g h^{(-1)}\\
        &= \frac{1+\beta}{2}\mathrm{tr}_g h^{(1)} + \frac{1-\beta}{2} \left(2\mathrm{tr}_gh^{(0)}-\mathrm{tr}_gh^{(1)}\right)\\
        &= \beta \mathrm{tr}_g h^{(1)} + (1-\beta)\mathrm{tr}_gh^{(0)} = 0,
    \end{align}
    which implies $h^{(\beta)}=0$ by Proposition~\ref{hypsur}, and hence we have $B^{(\beta)}=0$.
    Therefore, $M$ is a $\nabla^{(\beta)}$-autoparallel submanifold.\\
\end{proof}

\subsection{Doubly totally-umbilical submanifolds of Hessian manifolds}
Since Hessian manifolds are equipped with the Hessian curvature, it is worthwhile to prepare formulas for this curvature on doubly totally-umbilical submanifolds.
Hereafter, for a statistical immersion $\iota:(M,g,\nabla)\to(\widetilde{M},\widetilde{g},\widetilde{\nabla})$ we write $B^{*}=B^{(-1)}$, $\widehat{B}=B^{(0)}$, $H^{*}=H^{(-1)}$, $\widehat{H}=H^{(0)}$, $A^{*}=A^{(-1)}$, $\widehat{A}=A^{(0)}$, $\nabla^{\perp*}=\nabla^{\perp(-1)}$, and $\widehat{\nabla}^{\perp}=\nabla^{\perp(0)}$.\\

\begin{proposition}
    Let $\iota:(M,g,\nabla)\to(\widetilde{M},\widetilde{g},\widetilde{\nabla})$ be a statistical immersion, where $M$ is a doubly totally-umbilical submanifold.
    The following formulas hold.
    \begin{align}
        \widetilde{g}\left(\left(\widetilde{\nabla}_{\iota_*X}\widetilde{K}\right)(\iota_*Y,\iota_*Z),\iota_*W\right) =& g\left(\left(\nabla_X K\right)(Y,Z),W\right) + g(X,W)g(Y,Z)\widetilde{g}\left(H^{*},\widehat{H}-H\right)\\
        + g(X,Y)&g(Z,W)\widetilde{g}\left(H,H^{*}-\widehat{H}\right) + g(X,Z)g(Y,W)\widetilde{g}\left(H,H^{*}-\widehat{H}\right),\label{Hesscurv1}\\
        \widetilde{g}\left(\left(\widetilde{\nabla}_{\iota_*X}\widetilde{K}\right)(\iota_*Y,\iota_*Z),\xi\right) &= g(K_X Y,Z)\widetilde{g}(H^{*},\xi) + g(Y,Z)\widetilde{g}\left(\nabla^{\perp}_X H-\nabla^{\perp}_X\widehat{H},\xi\right)\\
        +g(X,Y)\widetilde{g}&\left(\widehat{\nabla}^{\perp}_Z H-\nabla^{\perp}_Z H,\xi\right) + g(X,Z)\widetilde{g}\left(\widehat{\nabla}^{\perp}_Y H-\nabla^{\perp}_Y H,\xi\right),\label{Hesscurv2}\\
        \widetilde{g}\left(\left(\widetilde{\nabla}_{\iota_*X}\widetilde{K}\right)(\iota_*Y,\xi),\iota_*Z\right) &= g(K_X Y,Z)\widetilde{g}(H^{*},\xi) + g(Y,Z)\widetilde{g}\left(\nabla^{\perp*}_X \widehat{H}-\nabla^{\perp*}_XH^{*},\xi\right)\\
        +g(X,Z)\widetilde{g}&\left(\nabla^{\perp*}_Y H^{*}-\widehat{\nabla}^{\perp}_YH^{*},\xi\right) + g(X,Y)\widetilde{g}\left(\widehat{\nabla}^{\perp}_Z H-\nabla^{\perp}_Z H,\xi\right),\label{Hesscurv3}\\
        \widetilde{g}\left(\left(\widetilde{\nabla}_{\iota_*X}\widetilde{K}\right)(\iota_*Y,\xi),\eta\right) &= \widetilde{g}\left(\left(\nabla_X K^{\perp}\right)(Y,\xi),\eta\right) + g(X,Y)\widetilde{g}\left(\widetilde{K}_{\xi}\eta,H\right)\\
         + g(X&,Y)\left(\widetilde{g}\left(\widehat{H},\xi\right)\widetilde{g}\Big(H^{*},\eta\Big) - \widetilde{g}\left(\widehat{H},\eta\right)\widetilde{g}\Big(H^{*},\xi\Big)\right)\label{Hesscurv4},
    \end{align}
    where $X,Y,Z,W\in\Gamma(TM)$, $\xi,\eta\in\Gamma(TM^{\perp})$, $K^{\perp} = \nabla^{\perp} - \widehat{\nabla}^{\perp}$, and
    \begin{align}
        \left(\nabla_X K^{\perp}\right)(Y,\xi) = \nabla^{\perp}_X K^{\perp}_Y\xi - K^{\perp}_{\nabla_X Y}\xi - K^{\perp}_Y \nabla^{\perp}_X \xi,
    \end{align}
    for $X,Y\in\Gamma(TM)$ and $\xi\in\Gamma(TM^{\perp})$.\\
\end{proposition}

\begin{proof}
    We prove the equations one by one.
    For $X,Y,Z,W\in\Gamma(TM)$, we have
    \begin{align}
        \widetilde{g}\left(\left(\widetilde{\nabla}_{\iota_*X}\widetilde{K}\right)(\iota_*Y,\iota_*Z),\iota_*W\right) =& \widetilde{g}\left(\widetilde{\nabla}_X \widetilde{K}_{\iota_*Y}\iota_*Z - \widetilde{K}_{\widetilde{\nabla}_X \iota_* Y}\iota_*Z - \widetilde{K}_{\iota_*Y} \widetilde{\nabla}_X\iota_*Z, \iota_*W \right)\\
        =&\widetilde{g}\left(\widetilde{\nabla}_X\left(\iota_* K_Y Z + B(Y,Z) - \widehat{B}(Y,Z)\right),\iota_*W\right)\\
        -\widetilde{g}&\left(\widetilde{K}_{\iota_* Z}\left(\iota_* \nabla_X Y + B(X,Y)\right),\iota_*W\right)\\
        -\widetilde{g}&\left(\widetilde{K}_{\iota_*Y}\left(\iota_*\nabla_X Z + B(X,Z)\right),\iota_*W\right)\\
        =&\widetilde{g}\left(\iota_*\nabla_X K_Y Z - \iota_*A_{B(Y,Z)}X + \iota_*A_{\widehat{B}(Y,Z)}X,\iota_*W\right)\\
        &-\widetilde{g}\left(\iota_*K_{Z}\nabla_X Y-\iota_*A_{B(X,Y)}Z+\iota_*\widehat{A}_{B(X,Y)}Z,\iota_*W\right)\\
        &-\widetilde{g}\left(\iota_*K_{Y}\nabla_X Z-\iota_*A_{B(X,Z)}Y+\iota_*\widehat{A}_{B(X,Z)}Y,\iota_*W\right)\\
        =&g((\nabla_{X}K),(Y,Z)) + g(X,W)g(Y,Z)\widetilde{g}\left(H^{*},\widehat{H}-H\right)\\
        &+ g(X,Y)g(Z,W)\widetilde{g}\left(H,H^{*}-\widehat{H}\right)\\
        &+ g(X,Z)g(Y,W)\widetilde{g}\left(H,H^{*}-\widehat{H}\right),
    \end{align}
    where we used \eqref{umcon} to obtain the last row.\\

    For $X,Y,Z\in\Gamma(TM)$ and $\xi\in\Gamma(TM^{\perp})$, we have
    \begin{align}
        \widetilde{g}\left(\left(\widetilde{\nabla}_{\iota_*X}\widetilde{K}\right)(\iota_*Y,\iota_*Z),\xi\right) =& \widetilde{g}\left(\widetilde{\nabla}_X\left(\iota_* K_Y Z + B(Y,Z) - \widehat{B}(Y,Z)\right),\xi\right)\\
        -\widetilde{g}&\left(\widetilde{K}_{\iota_* Z}\left(\iota_* \nabla_X Y + B(X,Y)\right) + \widetilde{K}_{\iota_*Y}\left(\iota_*\nabla_X Z + B(X,Z)\right),\xi\right)\\
        =& \widetilde{g}\left(B(X,K_Y Z) + \nabla^{\perp}_XB(Y,Z) - \nabla^{\perp}_X\widehat{B}(Y,Z),\xi\right)\\
        &-\widetilde{g}\left(B(Z,\nabla_X Y) - \widehat{B}(Z,\nabla_X Y)+\nabla^{\perp}_Z B(X,Y) - \widehat{\nabla}^{\perp}_Z B(X,Y),\xi\right)\\
        &-\widetilde{g}\left(B(Y,\nabla_X Z) - \widehat{B}(Y,\nabla_X Z)+\nabla^{\perp}_Y B(X,Z) - \widehat{\nabla}^{\perp}_Y B(X,Z),\xi\right)\\
        =& g(X,K_Y Z)\widetilde{g}\left(H,\xi\right)\\
        &-(Xg(Y,Z)-g(\nabla_X Y,Z)-g(Y,\nabla_X Z))\widetilde{g}\left(\widehat{H}-H,\xi\right)\\
        &+g(Y,Z)\widetilde{g}\left(\nabla^{\perp}_X H - \nabla^{\perp}_X \widehat{H},\xi\right) - g(X,Y)\widetilde{g}\left(\nabla_Z^{\perp}H - \widehat{\nabla}_Z^{\perp}H,\xi\right)\\
        &-g(X,Z)\widetilde{g}\left(\nabla_Y^{\perp}H - \widehat{\nabla}_Y^{\perp}H,\xi\right)\\
        =& g(K_X Y,Z)\widetilde{g}\left(H,\xi\right)\\
        &+ 2g(K_X Y,Z)\widetilde{g}\left(\widehat{H}-H,\xi\right) + g(Y,Z)\widetilde{g}\left(\nabla^{\perp}_X H - \nabla^{\perp}_X \widehat{H},\xi\right)\\
        &+g(X,Y)\widetilde{g}\left(\widehat{\nabla}_Z^{\perp}H - \nabla_Z^{\perp}H,\xi\right) + g(X,Z)\widetilde{g}\left(\widehat{\nabla}_Y^{\perp}H - \nabla_Y^{\perp}H,\xi\right),
    \end{align}
    and since we have $H^{*} = 2\widehat{H} - H$, we obtain \eqref{Hesscurv2}.
    For $X,Y,Z\in\Gamma(TM)$ and $\xi\in\Gamma(TM^{\perp})$, we have
    \begin{align}
        \widetilde{g}\left(\left(\widetilde{\nabla}_{\iota_*X}\widetilde{K}\right)(\iota_*Y,\xi),\iota_*Z\right) =& \widetilde{g}\left(\widetilde{\nabla}_X \widetilde{K}_{\iota_*Y}\xi - \widetilde{K}_{\widetilde{\nabla}_X \iota_* Y}\xi - \widetilde{K}_{\iota_*Y} \widetilde{\nabla}_X\xi, \iota_*Z \right)\\
        =&\widetilde{g}\left(\widetilde{\nabla}_X\left(-\iota_*A_{\xi}Y+\iota_*\widehat{A}_{\xi}Y + \nabla^{\perp}_Y\xi-\widehat{\nabla}^{\perp}_Y\xi\right),\iota_* Z\right)\\
        -\widetilde{g}&\left(\widetilde{K}_{\xi}\left(\iota_*\nabla_X Y+B(X,Y)\right)+\widetilde{K}_{\iota_*Y}\left(-\iota_*A_{\xi}X+\nabla^{\perp}_X\xi\right),\iota_*Z\right)\\
        =&\widetilde{g}\left(\widetilde{\nabla}_X\left(-\widetilde{g}\left(H^{*}-\widehat{H},\xi\right)\iota_*Y\right) - \iota_*A_{\nabla^{\perp}_Y\xi}X+\iota_*A_{\widehat{\nabla}^{\perp}_Y\xi}X,\iota_*Z\right)\\
        +\widetilde{g}&\left(\iota_*A_\xi\nabla_X Y - \iota_*\widehat{A}_\xi\nabla_X Y - \widetilde{K}_{\xi}B(X,Y),\iota_*Z\right)\\
        + \widetilde{g}&\left(\widetilde{g}\left(H^{*},\xi\right)\widetilde{K}_{\iota_*Y}\iota_*X + \iota_*A_{\nabla^{\perp}_X\xi}Y - \iota_*\widehat{A}_{\nabla^{\perp}_X\xi}Y,\iota_*Z\right)\\
        =\widetilde{g}&\left(-X\widetilde{g}\left(H^{*}-\widehat{H},\xi\right)\iota_*Y - \widetilde{g}\left(H^{*},\nabla^{\perp}_Y\xi-\widehat{\nabla}^{\perp}_Y\xi\right)\iota_*X,\iota_*Z\right)\\
        &-\widetilde{g}\left(\widetilde{K}_{\xi}B(X,Y),\iota_*Z\right)\\
        &+\widetilde{g}\left(\widetilde{g}\left(H^{*},\xi\right)\iota_*K_X Y + \widetilde{g}\left(H^{*}-\widehat{H},\nabla^{\perp}_X\xi\right)\iota_*Y,\iota_*Z\right).
    \end{align}
    Here, it holds that
    \begin{align}
        \widetilde{g}\left(H^{*},\nabla^{\perp}_Y\xi-\widehat{\nabla}^{\perp}_Y\xi\right) &= Y\widetilde{g}\left(H^{*},\xi\right) - \widetilde{g}\left(\nabla^{\perp*}_YH^{*},\xi\right)\\
        &-Y\widetilde{g}\left(H^{*},\xi\right) + \widetilde{g}\left(\nabla^{\perp*}_YH^{*},\xi\right)\\
        &= \widetilde{g}\left(\widehat{\nabla}^{\perp}_YH^{*}-\nabla^{\perp*}_YH^{*},\xi\right),
    \end{align}
    and
    \begin{align}
        \widetilde{g}\left(\widetilde{K}_{\xi}B(X,Y),\iota_*Z\right) =& g(X,Y)\widetilde{g}\left(\widetilde{K}_{\iota_*Z}H,\xi\right)\\
        =& g(X,Y)\widetilde{g}\left(\nabla^{\perp}_Z H - \widehat{\nabla}^{\perp}_Z H,\xi\right),
    \end{align}
    thus we obtain \eqref{Hesscurv3}.
    Lastly, to prove \eqref{Hesscurv4}, for $X,Y\in\Gamma(TM)$ and $\xi,\eta\in\Gamma(TM^{\perp})$ we have
    \begin{align}
        \widetilde{g}\left(\left(\widetilde{\nabla}_{\iota_*X}\widetilde{K}\right)(\iota_*Y,\xi),\eta\right) =& \widetilde{g}\left(\widetilde{\nabla}_X\left(\widetilde{g}\left(\widehat{H}-H^{*},\xi\right)Y + K^{\perp}_Y\xi\right),\eta\right)\\
        -\widetilde{g}&\left(\widetilde{K}_{\xi}\left(\iota_*\nabla_X Y+B(X,Y)\right)+\widetilde{K}_{\iota_*Y}\left(-\widetilde{g}\left(H^{*},\xi\right)X+\nabla^{\perp}_X\xi\right),\eta\right)\\
        =& \widetilde{g}\left(\widehat{H}-H^{*},\xi\right)\widetilde{g}\left(B(X,Y),\eta\right) + \widetilde{g}\left(\nabla_X^{\perp}K^{\perp}_Y\xi,\eta\right)\\
        &-\widetilde{g}\left(\nabla^{\perp}_{\nabla_X Y}\xi-\widehat{\nabla}^{\perp}_{\nabla_X Y}\xi,\eta\right) + g(X,Y)\widetilde{g}\left(\widetilde{K}_{\xi}\eta,H\right)\\
        &+g(X,Y)\widetilde{g}\left(H^{*},\xi\right)\widetilde{g}\left(H-\widehat{H},\eta\right) - \widetilde{g}\left(K^{\perp}_Y\nabla^{\perp}_X\xi,\eta\right)\\
        =& g(X,Y)\left(\widetilde{g}\left(\widehat{H},\xi\right)\widetilde{g}\Big(H,\eta\Big) - \widetilde{g}\Big(H^{*},\xi\Big)\widetilde{g}\left(\widehat{H},\eta\right)\right) \\
        &+ \widetilde{g}\left(\left(\nabla_{X}K^{\perp}\right)(Y,\xi),\eta\right) +  g(X,Y)\widetilde{g}\left(\widetilde{K}_{\xi}\eta,H\right).
    \end{align}
\end{proof}

\begin{cor}
    \label{hesssub}
    If $\iota:(M,g,\nabla)\to(\widetilde{M},\widetilde{g},\widetilde{\nabla})$ is a statistical immersion, where $M$ is a doubly totally-umbilical submanifold and $(\widetilde{M},\widetilde{g},\widetilde{\nabla})$ is a Hessian manifold of CHC $\widetilde{c}$, then the following conditions are equivalent$:$
    \renewcommand{\labelenumi}{$\operatorname{(\theenumi)}$}
    \begin{enumerate}
        \item The statistical manifold $(M,g,\nabla)$ is a Hessian manifold of CHC $c=\widetilde{c} - 4\widetilde{g}(\widehat{H},\widehat{H})$.
        \item The submanifold $M$ is $\nabla^{*}$-autoparallel.
    \end{enumerate}
\end{cor}

\begin{proof}
    Assume that $(M,g,\nabla)$ is a Hessian manifold of CHC $c=\widetilde{c} - 4\widetilde{g}(\widehat{H},\widehat{H})$.
    Since equation \eqref{chc} holds for $\nabla K$ and $\widetilde{\nabla}\widetilde{K}$, for an orthonormal pair $\{X,Y\}$ on $(M,g)$ we have
    \begin{align}
        \label{ceq1}
        0 = g((\nabla_X K)(Y,Y),X) = \widetilde{g}\left(H^{*},H-\widehat{H}\right),
    \end{align}
    by taking $\xi=H-\widehat{H}$ in \eqref{Hesscurv1}.
    Consequently, 
        \begin{align}
            \label{sizestar}
            \|H^*\|_{\widetilde{g}}^2 = \widetilde{g}\left(H^*,H\right) - 2\widetilde{g}\left(H^*,H-\widehat{H}\right)
        \end{align}
        is a constant function since $\widetilde{g}\left(H^*,H\right)$ is also constant.
        Thus, for each orthonormal pair $\{X,Y\}$ on $(M,g)$, from equation \eqref{Hesscurv3} we have
        \begin{align*}
            0 = g(K_Y Y,X)\widetilde{g}(H^*,H^*) - \widetilde{g}\left(\widehat{\nabla}_X^{\perp}H^*,H^*\right) = \widetilde{g}(K_Y Y,X)\|H^*\|_{\widetilde{g}}^2
        \end{align*}
        by taking $\xi = H^*$, which implies that $H^*=0$ or $K=0$.
        If $H^*=0$, then $M$ is a $\nabla^{*}$-autoparallel submanifold.
        If $K=0$, then $(M,g,\nabla)$ is a Hessian manifold of CHC $0$, which implies $\widetilde{c}=4\widetilde{g}(\widehat{H},\widehat{H})$.
        Here, equation \eqref{Hesscurv1} yields
        \begin{align}
            -4\widetilde{g}(\widehat{H},\widehat{H}) &= \widetilde{g}\left(\left(\widetilde{\nabla}_{\iota_*X}\widetilde{K}\right)\left(\iota_*X,\iota_*X\right),\iota_*X\right)\\
            &= 2\widetilde{g}\left(H,H^*-\widehat{H}\right)
        \end{align}
        for any unit vectors $X$ on $(M,g)$.
        We also have $\widetilde{g}(H^*,H)=\widetilde{g}(H^*,\widehat{H})$ from \eqref{ceq1}, hence
        \begin{align}
            0 &= 2\widetilde{g}(\widehat{H},\widehat{H}) + \widetilde{g}\left(H,H^*-\widehat{H}\right)\\
            &= \widetilde{g}(H,H^*) + \widetilde{g}(H^*,\widehat{H})\\
            &= 2\widetilde{g}(H,H^*),
        \end{align}
        which implies $H^*=0$ from \eqref{sizestar}, and again, $M$ is $\nabla^{*}$-autoparallel.
        The converse is follows immediately from \eqref{Hesscurv1}.
\end{proof}

\begin{example}
    Consider the doubly totally-umbilical submanifold $M = \{p\in\Delta^n\mid p(\omega)=b\}$ of the probability simplex $(\Delta^n,g^F,\nabla^{(\mathrm{e})})$ in Example~\ref{umprob}.
    The submanifold $M$ is also a $\nabla^{(\mathrm{m})}$-autoparallel submanifold.
    Since the probability simplex $(\Delta^n,g^F,\nabla^{(\mathrm{e})})$ has CHC $-1$, Corollary~\ref{hesssub} implies that $(M,g,\nabla)$ also a Hessian manifold of CHC.
    The Hessian curvature of $(M,g,\nabla)$ is $-(1-b)^{-1}$.
\end{example}

\section{Complete classification of doubly totally-umbilical submanifolds in the probability simplex}
We give a complete classification of doubly totally-umbilical submanifolds in the probability simplex.
From here on, we only consider submanifolds $M\subset\widetilde{M}$ of the ambient space and not immersed ones.\\
\begin{lemma}
    \label{umbequi}
    Let $M$ be a submanifold of $\Delta^n$.
    Then, the submanifold $M$ is a doubly totally-umbilical one of $(\Delta^n,g^F,\nabla^{\mathrm{(e)}})$ if and only if $\iota(M)$ is a doubly totally-umbilical submanifold of $((\mathbb{R}^+)^{n+1},g_0,D)$ where $\iota:\Delta^n\to(\mathbb{R}^+)^{n+1}$ is the embedding map in Example $\mathrm{\ref{probsimpemb}}$.
\end{lemma}
\begin{proof}
    The probability simplex $(\Delta^n,g^F,\nabla^{(\mathrm{e})})$ is a doubly totally-umbilical statistical submanifold of $((\mathbb{R}^+)^{n+1},g_0,D)$.
    If we let $\iota_0:M\to\Delta^n$ the inclusion map, $(g,\nabla)$ the induced statistical structure on $M$, $B$ the second fundamental form of $\iota_0$ with respect to $\nabla^{(\mathrm{e})}$, and $\widetilde{H}$ the mean curvature of $\iota$ with respect to $D$, we have
    \begin{align*}
        D_X(\iota\circ\iota_0)_*Y &= \iota_*(\nabla^{(\mathrm{e})}_X\iota_{0*}Y) + g^F(\iota_{0*}X,\iota_{0*}Y)\widetilde{H}\\
        &= (\iota\circ\iota_0)_*(\nabla_X Y) + \iota_*B(X,Y) + g(X,Y)\widetilde{H},
    \end{align*}
    thus the equivalence holds.\\
\end{proof}
From Lemma \ref{umbequi}, it is sufficient to classify the doubly totally-umbilical submanifolds in the denormalized state space first.

\subsection{Doubly totally-umbilical submanifolds in the denormalized state space}
We first classify doubly autoparallel submanifolds in the denormalized state space.\\

\begin{proposition}
    \label{dapdenorm}
    A submanifold $M^m$ of the denormalized state space $((\mathbb{R}^+)^n,g_0,D)$ is doubly autoparallel if and only if $M$ is contained in 
    \begin{equation}
        \label{dapde}
        \mathcal{P}^{m} = \left\{(\eta^1,\ldots,\eta^n)\in(\mathbb{R}^+)^n \mid \eta^{i_k} = a_{i_k}\eta^{\phi(i_k)},\,\eta^{j_l} = b_{j_l}\right\}
    \end{equation}
    where $i_1,\ldots,i_{\mathcal{A}},j_1,\ldots,j_{\mathcal{B}}\in\{1,\ldots,n\}$ are $\mathcal{A}+\mathcal{B}=n-m$ distinct integers, $\phi:\{i_1,\ldots,i_{\mathcal{A}}\}\to(\{1,\ldots,n\}\setminus\{i_1,\ldots,i_{\mathcal{A}},j_1,\ldots,j_{\mathcal{B}}\})$, and $a_{i_1},\ldots,a_{i_{\mathcal{A}}},b_{j_1},\ldots,b_{j_{\mathcal{B}}}\in\mathbb{R}^+$.
\end{proposition}
\begin{proof}
    For simplicity, suppose that $M$ is contained in \eqref{dapde}, where $i_k=m+k$, and $j_l=m+A+l$.
    Let $(g,\nabla)$ be the statistical structure induced on $M$ by $(g_0,D)$ and the inclusion map $\iota:M\to(\mathbb{R}^+)^n$.
    Since $(\eta^1,\ldots,\eta^n)$ is an affine coordinate system of $D^*$, it is clear that $M$ is a $D^*$-autoparallel submanifold of $(\mathbb{R}^+)^n$.
    Define a coordinate system $(\xi^1,\ldots,\xi^m)$ on $M$ by $(\xi^1(p),\ldots,\xi^m(p))=(\eta^1(p),\ldots,\eta^m(p)),\,p\in M$.
    By using these coordinate systems, we have
    \begin{align*}
        \iota_*\frac{\partial}{\partial \xi^i} &= \frac{\partial}{\partial \eta^i} + \sum_{k\in\phi^{-1}(i)}a_k\frac{\partial}{\partial \eta^{k}}
    \end{align*}
    and
    \begin{align*}
        D_{\frac{\partial}{\partial \xi^i}}\iota_*\frac{\partial}{\partial \xi^j} &= -\frac{\delta_{ij}}{\eta^i}\frac{\partial}{\partial \eta^i} - \sum_{k\in\phi^{-1}(i)\cap\phi^{-1}(j)}\frac{(a_k)^2}{\eta^k}\frac{\partial}{\partial \eta^k}\\
        &= -\frac{\delta_{ij}}{\xi^i}\frac{\partial}{\partial \eta^i} - \delta_{ij}\sum_{k\in\phi^{-1}(i)}\frac{a_k}{\xi^i}\frac{\partial}{\partial \eta^k}\\
        &= -\frac{\delta_{ij}}{\xi^i}\iota_*\frac{\partial}{\partial \xi^i}.
    \end{align*}
    Since this a linear combination of $\iota_*\frac{\partial}{\partial \xi^i}$, the second fundamental form with respect to $D$ is $0$, hence $M$ is a doubly autoparallel submanifold of $((\mathbb{R}^+)^n,g_0,D)$.\\

    Now we assume that $M$ is a doubly autoparallel submanifold of $((\mathbb{R}^+)^n,g_0,D)$.
    Since $M$ is a $D^*$-autoparallel submanifold of $(\mathbb{R}^+)^n$, it is an open subset of $P^{m}\cap(\mathbb{R}^+)^n$ where $P^{m}\subset\mathbb{R}^n$ is an $m$-dimensional plane with respect to the $(\eta^1,\ldots,\eta^n)$ coordinates.
    After a suitable reordering of the elements of the $D^*$-affine coordinates $(\eta^1,\ldots,\eta^n)$, there exists a global coordinate system $(\xi^1,\ldots,\xi^m)$ on $P^{m}$ such that
    \begin{align}
        \begin{pmatrix} \eta^1\\ \vdots \\ \vdots \\ \eta^n \end{pmatrix} = 
        \begin{pmatrix} 
         1 & 0 & \dots & 0 \\
         0 & 1 & &0  \\
         \vdots & & \ddots & \vdots \\
         0 & 0 & \dots & 1\\
         a_{m+1,1} & a_{m+1,2} & \dots & a_{m+1,m}\\
         a_{m+2,1} & a_{m+2,2} & \dots & a_{m+2,m}\\
         \vdots & &\ddots & \vdots\\
         a_{n-m,1} & a_{n-m,2} & & \dots a_{n-m,m}
        \end{pmatrix} 
        \begin{pmatrix} \xi^1\\ \vdots \\ \xi^m \end{pmatrix} 
        + \begin{pmatrix} 0\\ \vdots \\ 0 \\ b_{m+1} \\ \vdots \\ b_{n-m} \end{pmatrix}.
    \end{align}
    Here, we have
    \begin{align}
        \label{4eq1}
        \iota_*\frac{\partial}{\partial \xi^i} = \frac{\partial}{\partial \eta^i} + \sum_{k=1}^{n-m}a_{k,i}\frac{\partial}{\partial \eta^k}
    \end{align}
    and
    \begin{align}
        \label{4eq2}
        D_{\frac{\partial}{\partial \xi^i}}\iota_*\frac{\partial}{\partial \xi^j} &= -\frac{\delta_{ij}}{\xi^i}\frac{\partial}{\partial \eta^i} - \sum_{k=1}^{n-m}\frac{a_{k,i}a_{k,j}}{\sum_{l=1}^m a_{k,l}\xi^l + b_k}\frac{\partial}{\partial \eta^k}.
    \end{align}
    Since \eqref{4eq2} can be expressed as a linear combination of $\left\{\iota_*\frac{\partial}{\partial \xi^1},\ldots,\iota_*\frac{\partial}{\partial \xi^m}\right\}$, we have $a_{k,i}a_{k,j}=0$ for each $i\neq j$.
    Thus, for each $k\in\{1,\ldots,n-m\}$ there exists a unique $1\leq i_k\leq m$ such that $a_{i,k}=0$ for each $i\neq i_k$.
    Here, from equations \eqref{4eq1} and \eqref{4eq2}, it follows that
    \begin{align}
        \frac{a_{k,i_k}}{\xi^{i_k}} = \frac{(a_{k,i_k})^2}{a_{k,i_k}\xi^{i_k}+b_{k}}.
    \end{align}
    If $a_{k,i_k}\neq 0$, then $b_k=0$ hold, and consequently we have $a_{k,i_k}>0$.
    If $a_{k,i_k}= 0$, then $b_k=\eta^k>0$.
    Therefore, the manifold $\mathcal{P}^m = P^m\cap(\mathbb{R}^+)^n$ can be expressed in terms of equation \eqref{dapde}.\\
\end{proof}

\begin{remark}
    \label{dapcoor}
    In Proposition \ref{dapdenorm}, the induced statistical structure $(g,\nabla)$ on $M$ is expressed by
    \begin{align*}
        g_0\left(\frac{\partial}{\partial \xi^i},\frac{\partial}{\partial \xi^j}\right) &= \left(1+\sum_{k\in\phi^{-1}(i)}a^k\right)\frac{\delta_{ij}}{\xi^i}\\
        \nabla_{\frac{\partial}{\partial \xi^i}}\frac{\partial}{\partial \xi^j} &= -\frac{\delta_{ij}}{\xi^i}.
    \end{align*}
    If we define a coordinate system $(\sigma^1,\ldots,\sigma^m)$ on $M$ by $\sigma^i=(1+\sum_{k\in\phi^{-1}(i)}a^k)\xi^i$, then the expressions of $(g,\nabla)$ with respect to $(\sigma^1,\ldots,\sigma^m)$ coincide with the expression of $(g_0,D)$ given by \eqref{denorprobsimp}.
    Thus, we can regard $M$ as an open statistical submanifold of $((\mathbb{R}^+)^m,g_0,D)$.\\
\end{remark}

The doubly totally-umbilical submanifolds in the denormalized state space of codimension one that are not doubly autoparallel is already classified in \cite{Furuhata2024Toward,FURUHATA2011S86}.
We provide a different proof to this classfication.\\

\begin{proposition}
    \label{dtudenorm1}
    A submanifold $M^m$ of the denormalized state space $((\mathbb{R}^+)^{m+1},g_0,D)$ is doubly totally-umbilical if and only if it is contained in one of the following sets$\mathrm{:}$
    \renewcommand{\labelenumi}{$\mathrm{(\theenumi)}$}
    \begin{enumerate}
        \item The set $\mathcal{P}^m$ in $\eqref{dapde}$.
        \item The set $\Delta^m(b) = \{(\eta^1,\ldots,\eta^{m+1})\in(\mathbb{R}^+)^{m+1}\mid \sum_{i=1}^{m+1}\eta^i=b\}$, where $b\in\mathbb{R}^+$.
    \end{enumerate}
\end{proposition}

\begin{proof}
    Case $(1)$ follows immediately from Proposition \ref{dapdenorm}.
    If case $(2)$ holds, it can be proved that $\Delta^m(b)$ is a doubly totally-umbilical submanifold of $((\mathbb{R}^+)^{m+1},g_0,D)$ in the same manner as in Example \ref{probsimpisdtu}.\\
    
    Now we let $M^m$ be a doubly totally-umbilical submanifold of $((\mathbb{R}^+)^{m+1},g_0,D)$.
    We may assume that $M$ is not doubly autoparallel, since otherwise case $(1)$ follows from Proposition~\ref{dapdenorm}.
    Let $(g,\nabla)$ be the induced statistical structure by the inclusion $\iota:M\to(\mathbb{R}^+)^{m+1}$, and let $K,\widetilde{K}$ be the difference tensor field of $(g,\nabla)$, $(g_0,D)$, respectively.
    We prove $K\neq0$ on any point of $M$.
    If $K=0$ on some point $p\in M$, by \eqref{Hesscurv1}, for each orthonormal pair $X,Y\in T_p M$ we have
    \begin{align*}
        0 = g_0\left((D_{\iota_*X}\widetilde{K})(\iota_*Y,\iota_*Y),X\right) &= g_0\left(H^*,\widehat{H}-H\right),\\
        0 = g_0\left((D_{\iota_*X}\widetilde{K})(\iota_*X,\iota_*Y),Y\right) &= g_0\left(H,H^*-\widehat{H}\right),
    \end{align*}
    since $((\mathbb{R}^+)^{m+1},g_0,D)$ is a Hessian manifold of CHC $0$.
    Thus it must hold that $H_p=H_p^*=\widehat{H}_p$.
    From this result, on $p$ we have 
    \begin{align*}
        0 = \iota_*K_X Y &=\iota_*\nabla_X Y - \iota_*\nabla^g_X Y\\
        &= D_X\iota_* Y - D^{g_0}_X\iota_* Y - g(X,Y)(H-\widehat{H}) = \widetilde{K}_{\iota_* X}\iota_* Y
    \end{align*}
    for $X,Y\in\Gamma(TM)$.
    Thus, if $X\in T_{\iota(p)}(\mathbb{R}^{+})^{m+1}$ is tangent to $M$, then it holds that $\widetilde{K}_X X=0$.
    However, if we express $X$ with respect to $(\eta^1,\ldots,\eta^{m+1})$ by $X = \sum_{i=1}^{m+1}X^i\frac{\partial}{\partial \eta^i}$, we have
    \begin{align*}
        0 = \widetilde{K}_X X = -\sum_{i=1}^{m+1} \left(X_i\right)^2\frac{1}{\eta^i}\frac{\partial}{\partial \eta^i},
    \end{align*}
    which implies that $X^i=0$ for each $i$ and consequently $X=0$.
    This contradicts that the dimension of $M$ is $m\geq 2$, therefore we have $K\neq0$ on any point of $M$.\\

    For each $X,Y,Z\in\Gamma(TM)$, from \eqref{Hesscurv2} we have
    \begin{align*}
        0 =& g(K_X Y,Z)g_0(H^*,\boldsymbol{n}) - g(Y,Z)g_0\left(D^{\perp}_X \widehat{H},\boldsymbol{n}\right)\\
        & + g(X,Y)g_0\left(\widehat{D}^{\perp}_ZH,\boldsymbol{n}\right) + g(X,Z)\left(\widehat{D}^{\perp}_Y H,\boldsymbol{n}\right)
    \end{align*}
    since $D^{\perp}H = D^{\perp*}H^* = \widehat{D}^{\perp}\widehat{H} = 0$ holds from \eqref{parmean}, where $\boldsymbol{n}$ is the normal vector field of $\iota:M^m\to(\mathbb{R}^+)^{m+1}$.
    Here, from Proposition \ref{umhyp3} and Corollary \ref{umhyp4}, we have $D^{\perp} \widehat{H}=\widehat{D}^{\perp} H=0$.
    Thus $H^*=0$, therefore $M$ is a $D^*$-autoparallel submanifold of $(\mathbb{R}^+)^{m+1}$, which is contained in a hyperplane with respect to the coordinate system $(\eta^1,\ldots,\eta^{m+1})$. 
    There exists $a_1,\ldots,a_{m+1},b\in\mathbb{R}$ such that $(a_1,\ldots,a_{m+1})\neq(0,\ldots,0)$ and $M$ is contained in
    \begin{align}
        \label{planeform}
        \left\{(\eta^1,\ldots,\eta^{m+1})\mid \sum_{i=1}^{m+1}a_i\eta^i = b\right\}.
    \end{align}
    Here, if we set $\varphi(\eta^1,\ldots,\eta^{m+1}) = \sum_{i=1}^{m+1}a_i\eta^i - b$, for each $p\in M$ and $X=\sum_{i=1}^{m+1}X^i\frac{\partial}{\partial \eta^i}\in T_{\iota(p)}(\mathbb{R}^+)^{m+1}$, the equation $0 = X\varphi = \sum_{i=1}^{m+1}X^ia_i$ is equivalent to $X\in\iota_*T_pM$, and we have
    \begin{align*}
        \boldsymbol{n}_p = \frac{1}{\|(\operatorname{grad}_{g_0}\varphi)_{\iota(p)}\|_{g_0}}(\operatorname{grad}_{g_0}\varphi)_{\iota(p)},\quad p\in M,
    \end{align*}
    where $\operatorname{grad}_{g_0}\varphi$ is the gradient vector field of $\varphi$ on $((\mathbb{R}^+)^{m+1},g_0)$.
    Particularly, by $g_0 = \sum_{i=1}^{m+1}(\eta^i)^{-1}d\eta^i\otimes d\eta^i$ on the coordinates $(\eta^1,\ldots,\eta^{m+1})$, we have
    \begin{align*}
        \operatorname{grad}_{g_0}\varphi &= \sum_{i=1}^{m+1}a_i\eta^i\frac{\partial}{\partial \eta^i},
    \end{align*}
    which implies $X\|\operatorname{grad}_{g_0}\varphi\|_{g_0} = 0$ for every $X\in\iota_*(T_pM)$.
    From Corollary \ref{umhyp4} and equation \eqref{parmean}, we also have
    \begin{align*}
        0 &= g_0(D^*_X \boldsymbol{n},\boldsymbol{n})\\
        &= \frac{1}{\|\operatorname{grad}_{g_0}\varphi\|_{g_0}}g_0\left(D^*_X \operatorname{grad}_{g_0}\varphi, \boldsymbol{n}\right)\\
        &= \sum_{i=1}^{m+1}X^ia_i^2
    \end{align*}
    for each $X\in\iota_*T_pM$.
    This implies that there exists $t\neq0$ such that $(ta_1,\ldots,ta_{m+1})=(a_1^2,\ldots,a_{m+1}^2)$, thus for each $1\leq i \leq m+1$ we have $a_i=t$ or $a_i=0$.
    We can assume $t=1$ by rescaling $b\in\mathbb{R}$ if necessary.
    Lastly, we prove that we have $a_i\neq 0$ for every $1\leq i \leq m+1$.
    If there exists an $1\leq i \leq m+1$ such that $a_i = 0$, then $\frac{\partial}{\partial \eta^i}\in\iota_*T_pM$.
    Here, from Proposition \ref{hypsur} we have
    \begin{align*}
        \frac{1}{m}\mathrm{tr}_gh\cdot \frac{\partial}{\partial \eta^{i}} = D^*_{\frac{\partial}{\partial \eta^{i}}}\boldsymbol{n} = 0,
    \end{align*}
    which implies $\mathrm{tr}_gh=0$.
    Since $m\|H\|_{g_0}=|\mathrm{tr}_g h|$, this contradicts $M$ is not doubly autoparallel.
    Hence we have $a_i\neq 0$, and the plane \eqref{planeform} is equal to $\Delta^{m}(b)$.\\
\end{proof}

Combining Propositions \ref{dapdenorm} and \ref{dtudenorm1}, we obtain a complete classification of doubly totally-umbilical submanifolds in the denormalized state space.\\

\begin{theorem}
    \label{dtudenorm2}
    A submanifold $M^m$ of the denormalized state space $((\mathbb{R}^+)^{n},g_0,D)$ is doubly totally-umbilical if and only if one of the following conditions holds$\mathrm{:}$
    \renewcommand{\labelenumi}{$\mathrm{(\theenumi)}$}
    \begin{enumerate}
        \item It is contained in the manifold $\mathcal{P}^m$ defined by $\eqref{dapde}$.
        \item There exists a doubly autoparallel submanifold $M\subset \mathcal{P}^{m+1}\subset (\mathbb{R}^+)^{n}$ such that, with the induced statistical structure $(\widetilde{g},\widetilde{\nabla})$ on $\mathcal{P}^{m+1}$, there exists a global $\widetilde{\nabla}^{*}$-affine chart $(\sigma^1,\ldots,\sigma^{m+1};(\mathbb{R}^+)^{m+1})$ on $\mathcal{P}^{m+1}$ such that $M$ is contained in $\Delta^m(b) = \{(\sigma^1,\ldots,\sigma^{m+1})\in(\mathbb{R}^+)^{m+1}\mid \sum_{i=1}^{m+1}\sigma^i=b\}$, where $b\in\mathbb{R}^+$.
    \end{enumerate}
\end{theorem}
    \begin{proof}
        If $(1)$ holds, it is clear that $M$ is a doubly autoparallel submanifold of $((\mathbb{R}^+)^{n},g_0,D)$ by Proposition \ref{dapdenorm} and hence doubly totally-umbilical.
        Suppose that $(2)$ holds.
        By Proposition \ref{dtudenorm1} the manifold $M$ is a doubly totally-umbilical submanifold of $(\mathcal{P}^{m+1},\widetilde{g},\widetilde{\nabla})$.
        With the inclusion maps $\iota_1:M\to \mathcal{P}^{m+1}$, $\iota_2:\mathcal{P}^{m+1}\to (\mathbb{R}^+)^{n}$, we define $\iota=\iota_2\circ\iota_1$, and let $(g,\nabla)$ be the statistical structure induced on $M$ by $\iota_1$.      
        Since $\mathcal{P}^{m+1}$ is doubly autoparallel in $((\mathbb{R}^+)^{n},g_0,D)$, for every $\alpha\in\mathbb{R}$ and $X,Y\in\Gamma(TM)$ we have
        \begin{align*}
            D^{(\alpha)}_X \iota_*Y &= D^{(\alpha)}_X \iota_{2*}\iota_{1*}Y\\
            &=\iota_{2*}\widetilde{\nabla}^{(\alpha)}_X \iota_{1*}Y\\
            &=\iota_*\nabla^{(\alpha)}_X Y + g(X,Y)H^{(\alpha)}.
        \end{align*}
        Hence $M$ is a doubly totally-umbilical submanifold of $((\mathbb{R}^+)^{n},g_0,D)$.\\

        Conversely, let $M$ be a doubly totally-umbilical submanifold of $((\mathbb{R}^+)^{n},g_0,D)$.
        If $M$ is doubly autoparallel, then $(1)$ follows from Proposition~\ref{dapdenorm}.
        Therefore, we may assume that $M$ is not doubly autoparallel.
        Let $(g,\nabla)$ be the induced statistical structure on $M$ by $((\mathbb{R}^+)^{n},g_0,D)$ and the inclusion map $\iota:M\to(\mathbb{R}^+)^{n}$.
        As in the proof of Proposition~\ref{dtudenorm1}, the difference tensor $K$ of $(g,\nabla)$ does not vanish on any point of $M$.
        We first show that $\|H-\widehat{H}\|_{g_0}$ is constant.
        Since $g_0(H,H^{*})$ and $g_0(\widehat{H},\widehat{H})$ are constant and $2\widehat{H} = H + H^*$, we have
        \begin{align*}
            \|H-\widehat{H}\|^2_{g_0} &= \frac{1}{4}\|H-H^*\|^2_{g_0}\\
            &=\frac{1}{4}\left(\|H\|^2_{g_0}+\|H^*\|^2_{g_0}-2g_0\left(H,H^*\right)\right)\\
            &=\|\widehat{H}\|^2_{g_0} - g_0\left(H,H^*\right).
        \end{align*}
        Together with $\widehat{D}^{\perp}\widehat{H} = 0$, 
        \begin{align*}
            0 = X\|H-\widehat{H}\|^2_{g_0} = 2g_0\left(\widehat{D}_X^{\perp}H,H-\widehat{H}\right)
        \end{align*}
        holds for every $X\in\Gamma(TM)$.
        Thus, for each orthonormal pair $\{X,Y\}$ on $(M,g)$, from equation \eqref{Hesscurv2} we obtain
        \begin{align*}
            0 &= g(K_Y Y,X)g_0\left(H^*,H - \widehat{H}\right)
        \end{align*}
        by taking $\xi=H-\widehat{H}$.
        Since $K$ does not vanish at any point of $M$, we conclude that $\displaystyle g_0\left(H^*,H - \widehat{H}\right) = 0$.
        Consequently, 
        \begin{align*}
            \|H^*\|_{g_0}^2 = g_0\left(H^*,H\right) - 2g_0\left(H^*,H-\widehat{H}\right)
        \end{align*}
        is also constant.
        Thus, for each orthonormal pair $\{X,Y\}$ on $(M,g)$, from equation \eqref{Hesscurv3} we have
        \begin{align*}
            0 = g(K_Y Y,X)g_0(H^*,H^*) - g_0\left(\widehat{D}_X^{\perp}H^*,H^*\right) = g(K_Y Y,X)\|H^*\|_{g_0}^2
        \end{align*}
        by taking $\xi = H^*$.
        Again using $K\neq0$ on any point, we conclude that $H^*=0$, and hence $M$ is a $D^*$-autoparallel submanifold of $(\mathbb{R}^+)^n$.
        Fix a point $p\in M$.
        Since $\widehat{H}_p\neq 0$, the vector space $V=\iota_*T_p M\oplus\mathbb{R}\widehat{H}_p$ is $m+1$-dimensional.
        For every $X,Y\in \Gamma(TM)$, we have $(D^{g_0}_X \iota_*Y)_p,(D^{g_0}_X \widehat{H})_p\in V$.
        We identify $\mathbb{R}^{n+1}$ with $T_{0}\mathbb{R}^{n+1}$ by the $D^{g_0}$-affine coordinate $(y^1,\ldots,y^{n+1})$ where $0\in\mathbb{R}^{n+1}$ is the origin, so that we have $V\subset\mathbb{R}^n$
        The set $\mathcal{P}^{m+1} = (\mathbb{R}^+)^{n}\cap\{p+v\mid v\in V\}$ is an affine $(m+1)$-plane of the affine space $((\mathbb{R}^+)^{n},D^{g_0})$ containing $M$.
        On the other hand, for each $X,Y\in \Gamma(TM)$, we have $(D_X \iota_*Y)_p,(D_X \widehat{H})_p=2(D_X H)_p\in V$, so $\mathcal{P}^{m+1}$ is also an affine $(m+1)$-plane of the affine space $((\mathbb{R}^+)^{n},D)$.
        Hence $\mathcal{P}^{m+1}$ is a doubly autoparallel submanifold of $((\mathbb{R}^+)^{n},g_0,D)$.
        If we denote by the induced statistical structure $(\widetilde{g},\widetilde{\nabla})$, it follows that $M$ is a doubly totally-umbilical submanifold of $(\mathcal{P}^{m+1},\widetilde{g},\widetilde{\nabla})$.
        Using the global chart $(\sigma^1,\ldots,\sigma^{m+1})$ from Remark \ref{dapcoor}, the statistical manifold $(\mathcal{P}^{m+1},\widetilde{g},\widetilde{\nabla})$ can be identified as $((\mathbb{R}^+)^{m+1},g_0,D)$.
        Thus, Proposition \ref{dtudenorm1} implies that there exists a $b\in\mathbb{R}^+$ such that $M$ is contained in $\Delta^m(b)$ defined in the coordinates $(\sigma^1,\ldots,\sigma^{m+1})$.\\
    \end{proof}

\subsection{Doubly totally-umbilical submanifolds in the probability simplex}
Using Theorem \ref{dtudenorm2}, we classify doubly totally-umbilical submanifolds in the probability simplex.\\

\begin{theorem}
    \label{maint}
    A submanifold $M^m$ of the probability simplex $(\Delta^n,g^F,\nabla^{\mathrm{(e)}})$ is doubly totally-umbilical if and only if it is contained in the following $\nabla^{\mathrm{(m)}}$-autoparallel submanifold$\mathrm{:}$
    \begin{align}
        \label{dtusub}
        \Xi = \Xi_{a_{i_1},\ldots,a_{i_{\mathcal{A}}},b_{j_1},\ldots,b_{j_{\mathcal{B}}},\phi} = \left\{p\in\Delta^n\mid p(\omega_{i_k}) = a_{i_k}p(\omega_{\phi(i_k)}),\,p(\omega_{j_l})=b_{j_l}\right\},
    \end{align}
    where $i_1,\ldots,i_{\mathcal{A}},j_1,\ldots,j_{\mathcal{B}}\in\{1,\ldots,n+1\}$ are $\mathcal{A}+\mathcal{B}=n-m$ distinct integers, $\phi:\{i_1,\ldots,i_{\mathcal{A}}\}\to(\{1,\ldots,n+1\}\setminus\{i_1,\ldots,i_{\mathcal{A}},j_1,\ldots,j_{\mathcal{B}}\})$, $a_{i_1},\ldots,a_{i_{\mathcal{A}}}\in\mathbb{R}^+$, and $b_{j_1},\ldots,b_{j_{\mathcal{B}}}\in(0,1)$.
    In particular, if we denote by $(g,\nabla)$ the induced statistical structure on $\Xi$, the following hold $\mathrm{:}$
    \renewcommand{\labelenumi}{$\mathrm{(\theenumi)}$}
    \begin{enumerate}
        \item The submanifold $M$ is doubly autoparallel if and only if $\mathcal{B} = 0$.
        If $M$ is doubly autoparallel, then there exists a statistical diffeomorphism $f:(\Delta^m,g^F,\nabla^{(\mathrm{e})})\to(\Xi,g,\nabla)$ such that $\iota\circ f$ is a Markov embedding, where $\iota:\Xi\to\Delta^n$ is the inclusion map.\\
        
        \item If $\mathcal{B}\neq 0$, there exists a Markov embedding $f:\Delta^{m+1}\to\Delta^n$ such that $\Xi\subset f(\Delta^{m+1})$, and if we identify $f(\Delta^{m+1})$ with $\Delta^{m+1}$, then $(\Xi,g,\nabla)$ is a doubly totally-umbilical statistical submanifold of $(\Delta^{m+1},g^F,\nabla^{\mathrm{(e)}})$.\\
        
        \item The affine connection $\nabla|_{M}$ is complete if and only if $M=\Xi$.
    \end{enumerate}
\end{theorem}

We note that the affine connection $\nabla^{(\mathrm{e})}$ is complete on $\Delta^n$, which will be proved in Example~\ref{compexp} of the appendix.\\

\begin{proof}
    Define $\Xi$ by \eqref{dtusub}, and let $\iota:\Delta^n\to(\mathbb{R}^+)^{n+1}$ be the embedding map in Example~\ref{probsimpemb}.
    Then $\Xi$ is a submanifold of $\mathcal{P}^{m+1}=P^{m+1}\cap(\mathbb{R}^+)^{n+1}$ where
    \begin{align}
        \label{dubhypplane}
        P^{m+1} = \left\{(\eta^1,\ldots,\eta^{n+1})\in\mathbb{R}^{n+1}\large\mid \eta^{i_k} = a_{i_k}\eta^{\phi(i_k)},\,\eta^{j_l} = b_{j_l}\right\},
    \end{align}
    and $(\eta^1,\ldots,\eta^{n+1})$ is the $D^*$-affine coordinate system defined in Example~\ref{probsimpemb}.
    If we define the global coordinate system $(\sigma^1,\ldots,\sigma^{m+1})$ in Remark~\ref{dapcoor} on $\mathcal{P}^{m+1}$, then
    \begin{align}
        \label{xirep}
        \Xi = \left\{(\sigma^1,\ldots,\sigma^{m+1})\in\mathcal{P}^{m+1} \mid \sum_{i=1}^{m+1}\sigma^i = 1 - \sum_{l=1}^{\mathcal{B}}b_{j_l}\right\}.
    \end{align}
    By Theorem~\ref{dtudenorm2} $(2)$, the manifold $M$ is a doubly totally-umbilical submanifold of $((\mathbb{R}^+)^{n+1},g_0,D)$, and by Lemma~\ref{umbequi} it is also a doubly totally-umbilical submanifold of $(\Delta^n,g^F,\nabla^{(\mathrm{e})})$.\\

    Conversly, assume that $M$ is a doubly totally-umbilical submanifold of $(\Delta^n,g^F,\nabla^{(\mathrm{e})})$.
    Let $\iota:\Delta^n\to(\mathbb{R}^+)^{n+1}$ the embedding in Example \ref{probsimpemb}, and identify $\Delta^n$ with $\iota(\Delta^n)$.
    Since $\Delta^n$ is a doubly totally-umbilical but not doubly autoparallel submanifold of $((\mathbb{R}^+)^{n+1},g_0,D)$, so is $M\subset(\mathbb{R}^+)^{n+1}$ by Lemma~\ref{umbequi}.
    By Theorem~\ref{dtudenorm2}, there exists an affine plane $P^{m+1}\subset\mathbb{R}^{n+1}$ of the Euclidean space $(\mathbb{R}^{n+1},g_0)$, such that $M^m\subset P^{m+1}$ and $\mathcal{P}^{m+1} = P^{m+1}\cap(\mathbb{R}^+)^{n+1}$ is a doubly autoparallel submanifold of $((\mathbb{R}^+)^{n+1},g_0,D)$.
    Thus, the affine plane $P^{m+1}$ is given by the form of \eqref{dubhypplane}.
    If we set $\Xi=\mathcal{P}^{m+1}\cap\Delta^n$, then it follows that $\Xi$ is an $m$-dimensional doubly totally-umbilical submanifold of $(\Delta^n,g^F,\nabla^{(\mathrm{e})})$ and $M\subset\Xi$ holds.\\

    In order to prove (1) and (2), we prove that $\Xi=\Xi_{a_{i_1},\ldots,a_{i_{\mathcal{A}}},b_{j_1},\ldots,b_{j_{\mathcal{B}}},\phi}$ is contained in the sphere $S^m(r)$ for some $0<r\leq2$.
    Indeed, on the standard coordinate system $(y^1,\ldots,y^{n+1})$ of $\mathbb{R}^{n+1}$, we have $\Delta^n=\{(y^1,\ldots,y^{n+1})\in(\mathbb{R}^+)^{n+1}\mid \sum_{i=1}^{n+1}(y^i)^2 = 4\}$, thus $\Xi\subset P^{m+1}\cap S^{n}(2)$.
    Again, we identify $\mathbb{R}^{n+1}$ with $T_{0}\mathbb{R}^{n+1}$ by the $D^{g_0}$-affine coordinate $(y^1,\ldots,y^{n+1})$, where $0\in\mathbb{R}^{n+1}$ is the origin.
    Let $p\in P^{m+1}$ be the point minimizing $\|p\|_{g_0}$ on $P^{m+1}$, and let $V=\{q-p\mid q\in P^{m+1}\}$.
    We have $P^{m+1} = \{p+v\mid v\in V\}$, and for each $q=p+v\in P^{m+1}\cap S^{n}(2)$, since $g_0(p,v)=0$ it holds that
    \begin{align}
    4=\|q\|_{g_0}^2=\|p\|_{g_0}^2+\|v\|_{g_0}^2.
    \end{align}
    Hence $P^{m+1}\cap S^{n}(2)$ is an Euclidean sphere $S^m(r)$ where $r=\sqrt{4-\|p\|_{g_0}^2}$.\\
    
    We prove (1).
    The sphere $S^m(r)$ is an autoparallel manifold of the Euclidean sphere $(S^n(2),\nabla^{g_0})$ if and only if $r=2$ (see \cite{Chen2019} for example).
    In this case, $p\in P^{m+1}$ is the origin $0\in\mathbb{R}^{n+1}$, which is equivalent to $\mathcal{B}=0$.
    Since $(\Delta^n,g^F)$ is an open Riemannian submanifold of $(S^n(2),g_0)$ and $M$ is already a $\nabla^{\mathrm{(m)}}$-autoparallel submanifold of $\Delta^n$, the manifold $M$ is a doubly autoparallel submanifold of $(\Delta^n,g^F,\nabla^{\mathrm{(e)}})$, if and only if $\mathcal{B}=0$.\\

    For the doubly autoparallel submanifold $\Xi=\Xi_{a_{i_1},\ldots,a_{i_{\mathcal{A}}},\phi}$, we can assume that $(i_1,\ldots,i_{\mathcal{A}})=(m+2,\ldots,n+1)$.
    Define $C_l=\{l\}\cup\phi^{-1}(l)$ and $Q_l:\Omega_{n+1}\to[0,\infty)$ by
    \begin{align}
        Q_l(k) =
        \begin{cases}
            1, &k\in\Omega_{m+1},\\
            a_k, &k\in\phi^{-1}(l),\\
            0, &k\notin\Omega_{m+1}\cup\phi^{-1}(l)
        \end{cases}
    \end{align}
    for each $l\in\Omega_{m+1}$.
    Then, if we define the diffeomorphism $f:\Delta^{m}\to\Xi$ as in equation \eqref{markovemb} by using $\{C_l\}$ and $\{Q_l\}$, the map $\iota\circ f$ is a Markov embedding, and $\iota\circ f:(\Delta^m,g^F,\nabla^{(\mathrm{e})})\to(\Delta^n,g^F,\nabla^{(e)})$ is a statistical immersion.
    Thus, the diffeomorphism $f:(\Delta^m,g^F,\nabla^{(\mathrm{e})})\to(\Xi,g,\nabla)$ is a statistical one.\\
    
    Now we prove (2).
    Suppose that $\mathcal{B}\neq 0$.
    Again, we assume that $(i_1,\ldots,i_{\mathcal{A}})=(m+2,\ldots,m+1+\mathcal{A})$ and $(j_1,\ldots,j_{\mathcal{B}})=(m+2+\mathcal{A},\ldots,n+1)$, by reordering the elements of $(\eta^1,\ldots,\eta^{n+1})$ if necessary.
    We identify $\mathbb{R}^{n+1}$ with $T_{0}\mathbb{R}^{n+1}$ by the $D^{*}$-affine coordinate $(\eta^1,\ldots,\eta^{n+1})$, where $0\in\mathbb{R}^{n+1}$ is the origin.
    Since $\mathcal{B}\neq0$, the point $p\in P^{m+1}$ minimizing $\|p\|_{g_0}$ is $p=(0,\ldots,0,b_{j_1},\ldots,b_{j_{\mathcal{B}}})$.
    We define vector spaces $V,W$ by $V = \{q-p\mid q\in P^{m+1}\}$ and
    \begin{align*}
        W = V\oplus\mathbb{R}p.
    \end{align*}
    By taking linear combinations of a basis of $W$, we obtain
    \begin{align*}
        W = \left\{(\eta^1,\ldots,\eta^{n+1})\in\mathbb{R}^{n+1}\large\mid \eta^{i_k} = a_{i_k}\eta^{\phi(i_k)},\,\eta^{j_l} = \frac{b_{j_l}}{b_{j_1}}\eta^{j_1}\right\}.
    \end{align*}
    Hence $\mathcal{W}=W\cap(\mathbb{R}^+)^{n+1}$ is a doubly autoparallel submanifold of $((\mathbb{R}^+)^{n+1},g_0,D)$.
    By $(1)$, the set $\Theta = \mathcal{W}\cap\Delta^{n}$ is a doubly autoparallel submanifold of $(\Delta^{n},g^F,\nabla^{(\mathrm{e})})$, and there exists a Markov embedding $f:\Delta^{m+1}\to\Delta^n$ such that $f(\Delta^{m+1})=\Theta$ holds.
    Since $\Xi\subset\Theta$ holds by construction, if we identify $\Theta$ by $\Delta^{m+1}$, then $(\Xi,g,\nabla)$ is a doubly totally-umbilical statistical submanifold of $(\Delta^{m+1},g^F,\nabla^{(\mathrm{e})})$.\\

    Lastly, we prove $(3)$.
    Using the global coordinate system $(\sigma^1,\ldots,\sigma^{m+1})$ in Remark~\ref{dapcoor} on $\mathcal{P}^{m+1}$, we identify $(\mathcal{P}^{m+1},\widetilde{g},\widetilde{\nabla})$ with the statistical submanifold $((\mathbb{R}^+)^{m+1},g_0,D)$.
    We also have that $\Xi$ is expressed by \eqref{xirep} on this global coordinate, so we define a coordinate system $(\rho^1,\ldots,\rho^m)$ on $\Xi$ by $(\rho^1(p),\ldots,\rho^m(p))=(\sigma^1(p),\ldots,\sigma^m(p)),\,p\in\Xi$.
    Then $(\rho^1,\ldots,\rho^m)$ is an affine coordinate system of the conjugate connection $\nabla^{*}$ of the statistical structure $(g,\nabla)$ on $\Xi$ induced by $(\widetilde{g},\widetilde{\nabla})$.
    If we define
    \begin{align}
        \psi(\rho^1,\ldots,\rho^m) = \sum_{i=1}^m\rho^{i}\log\rho^i + \left(1-\sum_{l=1}^{\mathcal{B}}b_{j_l}-\sum_{i=1}^m\rho^i\right)\log\left(1-\sum_{l=1}^{\mathcal{B}}b_{j_l}-\sum_{i=1}^m\rho^i\right)
    \end{align}
    on the coordinate system $(\rho^1,\ldots,\rho^m)$, we have $g=\nabla^*d\psi$.
    From Proposition \ref{comphess}, we obtain that $\nabla$ is complete on $\Xi$.
    Consequently $\nabla|_{M}$ is complete on $M$ if and only if $M=\Xi$.\\
\end{proof}

\begin{remark}
    Theorem~\ref{maint} also gives the complete classification of doubly autoparallel submanifolds in the probability simplex.
    A.~Ohara and H.~Ishi also gave this classification by an algebraic characterization \cite{ohara2018doubly}.\\
\end{remark}

\begin{remark}
    For the probability simplex $(\Delta^n,g^F,\nabla^{(\mathrm{e})})$, if we fix $1\leq\mathcal{B}\leq n-2$ distinct integers $j_1,\ldots,j_{\mathcal{B}}\in\{1,\ldots,n+1\}$, we obtain a foliation
    \begin{align*}
        \left\{\Xi_{b_{j_1},\ldots,b_{j_{\mathcal{B}}}}\mid (b_{j_1},\ldots,b_{j_{\mathcal{B}}})\in\Delta^{\mathcal{B}}\right\},
    \end{align*}
    of $\Delta^n$ where each leaf is a doubly totally-umbilical submanifold.
    Indeed, if we assume $n+1\notin\{j_1,\ldots,j_{\mathcal{B}}\}$ for simplicity, then we have $\eta^{j_l}=b_{j_l}$ on the $\nabla^{(\mathrm{m})}$-affine coordinate system $(\eta^1,\ldots,\eta^n)$ defined in Example~\ref{probsimp}.
    Foliations of statistical manifolds have been extensively studied in information geometry \cite{Uohashi2022, Gnandi2026, BoyomWolak2016}.\\
\end{remark}

\begin{appendices}
    \section{Completeness on Hessian manifolds}
    We assume that the manifold $M$ is connected throughout this appendix.\\
    \begin{definition}
        Let $\nabla$ be an affine connection on $M$.
        A $\nabla$\textit{-geodesic} is a curve $\gamma:(a,b)\to M$ satisfying
        \begin{align}
            \nabla_{\dot{\gamma}}\dot{\gamma}  = 0,
        \end{align}
        where $\dot{\gamma}$ is the velocity vector field of $\gamma$.
        The affine connection is called \textit{complete}, if for any $(p,v)\in TM$ there exists a $\nabla$-geodesic $\gamma:\mathbb{R}\to M$ such that $\gamma(0)=p$ and $\dot{\gamma}(0)=v$.\\
    \end{definition}

    In general, completeness of the affine connection $\nabla$ of a statistical manifold $(M,g,\nabla)$ is difficult to determine. 
    However if $(M,g,\nabla)$ is Hessian and there is a global affine coordinate system $(x^1,\ldots,x^m)$ with respect to $\nabla$, on $(x^1,\ldots,x^m)$ the $\nabla$-geodesics in this coordinate are given by
    \begin{align}
        x^i(\gamma(t)) = a^it + b^i,
    \end{align}
    where $a^1,\ldots,a^m,b^1,\ldots,b^m$ are constants.
    In particular, the completeness of $\nabla$ is equivalent to the coordinate map $(x^1,\ldots,x^m)$ is onto $\mathbb{R}^m$.
    Thus, the mixture connection $\nabla^{(\mathrm{m})}$ in Example~\ref{probsimp} is not complete.\\

    The following proposition characterizes completeness of the conjugate connection.\\
    \begin{proposition}
        \label{comphess}
        Let $(M,g,\nabla)$ be a simply connected Hessian manifold.
        Assume that there exists a global affine coordinate system $(x^1,\ldots,x^m)$ on $M$ with respect to $\nabla$ and a $\psi\in C^{\infty}(M)$ such that $g=\nabla d\psi$.
        Then the conjugate connection $\nabla^{*}$ is complete if and only if the map
        \begin{align}
            \label{diffeom}
            \Psi(p) = \left(\frac{\partial \psi}{\partial x^1}(p),\ldots,\frac{\partial \psi}{\partial x^m}(p)\right), \quad p\in M
        \end{align}
        is a diffeomorphism from $M$ to $\mathbb{R}^m$.\\
    \end{proposition}

    \begin{proof}
        Define functions
        \begin{align}
            (y^1,\ldots,y^m) = \left(\frac{\partial \psi}{\partial x^1}(p),\ldots,\frac{\partial \psi}{\partial x^m}(p)\right).
        \end{align}
        Then, $(y^1,\ldots,y^n)$ is a global affine coordinate system of $M$ with respect to $\nabla^{*}$ (see \cite{MR1800071} for example).
        Indeed, if we let
        \begin{align}
            g_{ij} = g\left(\frac{\partial}{\partial x^i},\frac{\partial}{\partial x^j}\right),
        \end{align}
        then the condition $g = \nabla d\psi$ implies
        \begin{align}
            \frac{\partial}{\partial x^i} = \sum_{l=1}^m \frac{\partial y^l}{\partial x^i}\frac{\partial}{\partial y^l} = \sum_{l=1}^m \frac{\partial^2 \psi}{\partial x^i \partial x^l}\frac{\partial}{\partial y^l} = \sum_{j=1}^m g_{il}\frac{\partial}{\partial y^l}.
        \end{align}
        Hence, we have
        \begin{align}
            \frac{\partial}{\partial y^i} = \sum_{l=1}^m g^{il}\frac{\partial}{\partial x^l},
        \end{align}
        where $g^{ij}$ is the $(i,j)$-element of the inverse matrix of $G=(g_{ij})$.
        For any vector field $X\in\Gamma(TM)$, we obtain
        \begin{align}
            g\left(\nabla^*_{X}\frac{\partial}{\partial y^i},\frac{\partial}{\partial x^j}\right) &= Xg\left(\frac{\partial}{\partial y^i},\frac{\partial}{\partial x^j}\right) - g\left(\frac{\partial}{\partial y^i},\nabla_X\frac{\partial}{\partial x^j}\right)\\
            &= X\left(\sum_{l=1}^mg^{il}g_{lj}\right) = X\delta_{ij} = 0.
        \end{align}
        Since $\left\{\frac{\partial}{\partial x^1},\ldots,\frac{\partial}{\partial x^m}\right\}$ is a frame of $TM$, it follows that $\nabla^{*}_{X}\frac{\partial}{\partial y^i}=0$ for all $X\in\Gamma(TM)$.
        Therefore the coordinate system $(y^1,\ldots,y^m)$ is an affine coordinate of $\nabla^{*}$.
        The claim now follows since $\nabla$ is complete if and only if $(y^1,\ldots,y^m)$ is a coordinate map onto $\mathbb{R}^{m}$.\\
    \end{proof}

     \begin{example}
        \label{compexp}
        Consider the probability simplex $(\Delta^n,g^F,\nabla^{(\mathrm{m})})$, in Example~\ref{probsimp}, but with the mixture connection.
        The function represented by 
        \begin{align}
            \psi(\eta^1,\ldots,\eta^n) = \sum_{i=1}^n\eta^{i}\log\eta^i + \left(1-\sum_{i=1}^n\eta^i\right)\log\left(1-\sum_{i=1}^n\eta^i\right)
        \end{align}
        on the affine coordinate $(\eta^1,\ldots,\eta^n)$ with respect to $\nabla^{(\mathrm{m})}$ satisfies $g^F=\nabla^{(\mathrm{m})}d\psi$.
        Moreover,
        \begin{align}
            \frac{\partial \psi}{\partial \eta^i} = \log\eta^i - \log\left(1-\sum_{l=1}^n\eta^l\right).
        \end{align}
        Thus the map given by \eqref{diffeom} is a diffeomorphism between $\{(\eta^1,\ldots,\eta^n)\in(\mathbb{R}^+)^n\mid \sum_{l=1}^n\eta^l < 1\}$ and $\mathbb{R}^n$, hence by Proposition~\ref{comphess} the exponential connection $\nabla^{(\mathrm{e})}$ is complete.
    \end{example}

\end{appendices}

\subsection*{Acknowledgments}
The author would like to express his gratitude to Professor Takashi Kurose and Professor Hitoshi Furuhata for their valuable support and insightful discussions during the course of this research.





\vspace{-1pt}
\bibliography{sn-bibliography}

@book{Chen2019,
  author    = {Chen, Bang-Yen},
  title     = {Geometry of Submanifolds},
  edition   = {2nd},
  series    = {Dover Books on Mathematics},
  publisher = {Dover Publications},
  address   = {Mineola, New York},
  year      = {2019},
  pages     = {192},
  isbn      = {0-486-83278-3,978-0-486-83278-4},
  url       = {https://store.doverpublications.com/products/9780486832784}
}

@article{Chen_1980, 
title={Classification of totally umbilical submanifolds in symmetric spaces}, 
volume={30}, DOI={10.1017/S1446788700016414}, 
number={2}, 
journal={Journal of the Australian Mathematical Society. Series A. Pure Mathematics and Statistics}, 
author={Chen, Bang-Yen},
year={1980},
}

@article{JIMENEZ2022101862,
title = {Umbilical submanifolds of {$H^k\times S^{n-k+1}$}},
journal = {Differential Geometry and its Applications},
volume = {81},
year = {2022},
issn = {0926-2245},
doi = {https://doi.org/10.1016/j.difgeo.2022.101862},
url = {https://www.sciencedirect.com/science/article/pii/S0926224522000158},
author = {M.I. Jimenez and R. Tojeiro}
}

@article{Usiraj,
author = {Uddin, Siraj and Özel, Cenap},
year = {2014},
title = {A classification theorem on totally umbilical submanifolds in a cosymplectic manifold},
volume = {43},
journal = {Hacettepe Journal of Mathematics and Statistics}
}

@article{b47c6540-6bcd-33b7-ac07-4e5c2b947949,
 ISSN = {12203874, 20650264},
 URL = {http://www.jstor.org/stable/43678609},
 author = {Cătălin Angelo IOAN},
 journal = {Bulletin mathématique de la Société des Sciences Mathématiques de Roumanie},
 publisher = {Societatea de Științe Matematice din România},
 title = {Totally Umbilical Lightlike Submanifolds},
 urldate = {2026-06-05},
 number = {87},
 volume = {39},
 year = {1996}
}

@article{10.21099,
author = {Yuichiro Sato},
title = {{Totally umbilical submanifolds in pseudo-Riemannian space forms}},
volume = {45},
journal = {Tsukuba Journal of Mathematics},
number = {2},
publisher = {Department of Mathematics, University of Tsukuba},
year = {2021},
doi = {10.21099/tkbjm/20214502097},
URL = {https://doi.org/10.21099/tkbjm/20214502097}
}

@article{Siddesha2019,
  author  = {Siddesha, M. S. and Bagewadi, C. S. and Nirmala, D.},
  title   = {Totally umbilical proper slant submanifolds of para-Kenmotsu manifold},
  journal = {Cubo},
  volume   = {21},
  number   = {2},
  year     = {2019},
  doi      = {10.4067/S0719-06462019000200041}
}

@article{FURUHATA2011S86,
title = {Statistical hypersurfaces in the space of {H}essian curvature zero},
journal = {Differential Geometry and its Applications},
volume = {29},
year = {2011},
issn = {0926-2245},
doi = {https://doi.org/10.1016/j.difgeo.2011.04.012},
author = {Hitoshi Furuhata},
}

@Inbook{Nay,
author="Ay, Nihat
and Jost, J{\"u}rgen
and L{\^e}, H{\^o}ng V{\^a}n
and Schwachh{\"o}fer, Lorenz",
title="Finite Information Geometry",
bookTitle="Information Geometry",
year="2017",
publisher="Springer International Publishing",
address="Cham",
isbn="978-3-319-56478-4",
doi="10.1007/978-3-319-56478-4_2",
}

@INPROCEEDINGS{8006748,
  author={Nagaoka, Hiroshi},
  booktitle={2017 IEEE International Symposium on Information Theory (ISIT)}, 
  title={Information-geometrical characterization of statistical models which are statistically equivalent to probability simplexes}, 
  year={2017},
  volume={},
  number={},
  doi={10.1109/ISIT.2017.8006748}}

@book{AmariNagaoka2000,
  author    = {Amari, Shun-ichi and Nagaoka, Hiroshi},
  title     = {Methods of Information Geometry},
  series    = {Translations of Mathematical Monographs},
  volume    = {191},
  publisher = {American Mathematical Society},
  address   = {Providence, RI},
  year      = {2000}
}

@article{Zhang2007ANO,
  title={A note on curvature of $\alpha$-connections of a statistical manifold},
  author={Jun Zhang},
  journal={Annals of the Institute of Statistical Mathematics},
  year={2007},
  volume={59},
}

@incollection{Ohara2017,
  author    = {Ohara, Atsumi},
  title     = {On Affine Immersions of the Probability Simplex and Their Conformal Flattening},
  booktitle = {Geometric Science of Information},
  series     = {Lecture Notes in Computer Science},
  volume      = {10589},
  publisher   = {Springer},
  address     = {Cham},
  year        = {2017},
  doi         = {10.1007/978-3-319-68445-1_29}
}

@article{HF2013,
  title={Hessian manifolds of nonpositive constant {H}essian sectional curvature},
  author={HITOSHI Furuhata and TAKASHI Kurose},
  journal={Tohoku Mathematical Journal, Second Series},
  volume={65},
  year={2013},
}

@article{Fujiwara2026Companion,
  author       = {Yoshitaka Fujiwara},
  title        = {A companion to the Nagaoka--Amari structure on the denormalized state space},
  journal      = {Information Geometry},
  year         = {2026},
  doi          = {10.1007/s41884-026-00204-8}
}

@article{Uohashi2022,
  author    = {Uohashi, Keiko},
  title     = {Extended Divergence on a Foliation by Deformed Probability Simplexes},
  journal   = {Entropy},
  year      = {2022},
  volume    = {24},
  number    = {12},
  doi       = {10.3390/e24121736}
}

@article{BoyomWolak2016,
  author    = {Michel Nguiffo Boyom and Robert A. Wolak},
  title     = {Transversely Hessian foliations and information geometry},
  journal   = {International Journal of Mathematics},
  year      = {2016},
  volume    = {27},
  number    = {11},
  doi       = {10.1142/S0129167X16500920}
}

@article{Gnandi2026,
  author    = {Emmanuel Gnandi and Michel Nguiffo Boyom and St{\'e}phane Puechmorel},
  title     = {Canonical foliations of statistical manifolds with statistical models},
  journal   = {Information Geometry},
  volume    = {9},
  year      = {2026},
  doi       = {10.1007/s41884-026-00193-8}
}

@book{shima2007geometry,
  title={The Geometry of {H}essian Structures},
  author={Shima, H.},
  isbn={9789812700315},
  lccn={2007298479},
  year={2007},
  Address={Singapore},
  publisher={World Scientific}
}

@book{MR1800071,
author = {Amari, Shun-ichi},
title = {Information Geometry and Its Applications},
year = {2018},
isbn = {4431567437},
publisher = {Springer Publishing Company, Incorporated},
PAGES = {xiii+374},
address = {Tokyo}
}

@inproceedings{ohara2018doubly,
  author    = {Ohara, Atsumi and Ishi, Hideyuki},
  title     = {Doubly Autoparallel Structure on the Probability Simplex},
  booktitle = {Information Geometry and Its Applications},
  series    = {Springer Proceedings in Mathematics \& Statistics},
  volume    = {252},
  year      = {2018},
  publisher = {Springer International Publishing},
  doi       = {10.1007/978-3-319-97798-0_12}
}

@book{cencov2000statistical,
  title={Statistical decision rules and optimal inference},
  author={Cencov, Nikolai Nikolaevich},
  number={53},
  year={2000},
  publisher={American Mathematical Society}
}

@article{amami,
author = {Aydin, Muhittin and Mihai, Adela and Mihai, Ion},
year = {2015},
title = {Some Inequalities on Submanifolds in Statistical Manifolds of Constant Curvature},
volume = {29},
number = {3},
journal = {Filomat},
doi = {10.2298/FIL1503465A}
}

@incollection{FuruhataHasegawa2016,
  author       = {Hitoshi Furuhata and Izumi Hasegawa},
  title        = {Submanifold Theory in Holomorphic Statistical Manifolds},
  booktitle    = {Geometry of Cauchy-Riemann Submanifolds},
  publisher    = {Springer},
  address      = {Singapore},
  year         = {2016},
  isbn         = {978-981-10-0915-0},
  doi          = {10.1007/978-981-10-0916-7_7},
}

@article{Satoh2020StatisticalSubmanifolds,
  author       = {Satoh, Naoto and Furuhata, Hitoshi and Hasegawa, Izumi and Nakane, Toshiyuki and Okuyama, Yukihiko and Sato, Kimitake and Shahid, Mohammad Hasan and Siddiqui, Aliya Naaz},
  title        = {Statistical submanifolds from a viewpoint of the Euler inequality},
  journal      = {Information Geometry},
  volume       = {4},
  year         = {2020},
  doi          = {10.1007/s41884-020-00032-4},
  publisher    = {Springer},
}

@article{Furuhata2024Toward,
  author    = {Hitoshi Furuhata},
  title     = {Toward differential geometry of statistical submanifolds},
  journal   = {Information Geometry},
  volume    = {7},
  year      = {2024},
  doi       = {10.1007/s41884-022-00075-9},
  url       = {https://rdcu.be/c0iJz},
}

@article {MR4452143,
    AUTHOR = {Kobayashi, Shimpei and Ohno, Yu},
     TITLE = {On a constant curvature statistical manifold},
   JOURNAL = {Information Geometry},
    VOLUME = {5},
      YEAR = {2022},
      ISSN = {2511-2481,2511-249X},
   MRCLASS = {53B12 (53C15)},
MRREVIEWER = {Mohammed\ Jamali},
doi = {10.1007/s41884-022-00065-x},
}

\end{document}